\documentclass[12pt,a4paper]{amsart}
\usepackage[totalwidth=15.75cm,totalheight=22.275cm]{geometry}
\usepackage[utf8]{inputenc}
\usepackage[font=small]{caption}
\usepackage{commath}
\usepackage{graphicx}
\usepackage{subcaption}
\usepackage{amsmath}     
\usepackage{amssymb}     
\usepackage{amsfonts}     
\usepackage{mathrsfs}     
\usepackage{amsthm}     
\usepackage[showonlyrefs]{mathtools}
\mathtoolsset{showonlyrefs}
\usepackage{hyperref}
\usepackage{color}     
\numberwithin{equation}{section}
\newtheorem{theorem}{Theorem}[section]
\newtheorem{lemma}[theorem]{Lemma}

\newtheorem{proposition}[theorem]{Proposition}
\newtheorem{thmx}{Theorem}

\theoremstyle{definition}
\newtheorem{definition}[theorem]{Definition}
\theoremstyle{remark}
\newtheorem{remark}{Remark}[section]
\newtheorem*{ack}{Acknowledgment}
\def\re{\operatorname{Re}}
\def\im{\operatorname{Im}}
\def\arg{\operatorname{arg}}
\def\sign{\operatorname{sign}}
\def\logarea{\operatorname{logarea}}
\def\supp{\operatorname{supp}}
\def\dist{\operatorname{dist}}
\def\R{\mathbb{R}}
\def\C{\mathbb{C}}
\def\N{\mathbb{N}}
\def\Z{\mathbb{Z}}
\def\D{\mathbb{D}}
\def\bC{\overline{\C}}
\def\LO{O}

\def\Tan{\mathrm{T}}
\def\Broom{\mathrm{B}}
\newcommand{\eps}{\varepsilon}

\title[Differential equations with solutions having only real zeros]{Second order linear differential equations with a basis
of solutions having only real zeros}
\author{Walter Bergweiler}
\address{Mathematisches Seminar, Christian--Albrechts--Uni\-versi\-t\"at zu Kiel, 
24098 Kiel, Germany}
\email{bergweiler@math.uni-kiel.de}
\author{Alexandre Eremenko}
\address{Department of Mathematics, Purdue University,
West Lafayette, IN 47907, USA}
\email{eremenko@math.purdue.edu}
\author{Lasse Rempe}
\address{Department of Mathematical Sciences, University of Liverpool, Liverpool L69 7ZL, UK}
\email{l.rempe@liverpool.ac.uk}
\dedicatory{To the victims of the Russian aggression in Ukraine}

\makeatletter
\@namedef{subjclassname@2020}{%
  \textup{2020} Mathematics Subject Classification}
\makeatother


\begin{document}

\begin{abstract}
Let $A$ be a transcendental entire function of finite order.
We show that if the differential equation $w''+Aw=0$ has two 
linearly independent solutions with only real zeros, then 
the order of $A$ must be an odd integer or one half of an odd integer.
Moreover, $A$ has completely regular growth in the sense of Levin
and Pfluger. These results follow from a more general geometric theorem, which classifies symmetric local homeo\-morphisms from the plane to the sphere
for which all zeros and poles lie on the real axis, and which have only finitely many singularities over finite non-zero values.
\end{abstract}

\maketitle

\section{Introduction and results} \label{intro}
For any entire function~$A$, all solutions of the differential equation
\begin{equation}\label{1}
w''+Aw=0
\end{equation}
are entire.
We consider the question of when this equation has two
linearly independent solutions 
which have only real zeros.  For a polynomial $A$ this is possible only when $A$
is constant \cite[Theorem~3]{HSW}.  On the other hand, there are transcendental 
coefficients $A$ for which this happens. However, we
will show that if $A$ has finite order, then this is possible only 
in special cases. In particular, the order must be an odd integer or
one half of an odd integer.

We begin by making some general remarks on the equation~\eqref{1},
all of which can be found in~\cite{Laine1993}.
Let $w_1$ and $w_2$ be two linearly independent solutions. Then
their Wronskian determinant
\begin{equation}\label{Wronskian}
W=W(w_1,w_2)=w_1 w_2'-w_1' w_2
\end{equation}
is constant. A pair of solutions
$(w_1,w_2)$ is called {\em normalized} if $W=1$.

The ratio of two linearly independent solutions
$F=w_2/w_1$ satisfies the Schwarz equation
\begin{equation}\label{Sch}
\frac{F'''}{F'}-\frac{3}{2}\left(\frac{F''}{F'}\right)^2=2A.
\end{equation}
The function $F$ is meromorphic in $\C$ and locally univalent. All
meromorphic locally univalent meromorphic functions arise in this way.
A normalized pair can be recovered from $F$ by the formulae
\[
w_1^2=\frac{1}{F'},\quad w_2^2=\frac{F^2}{F'}.
\]

The product $E=w_1w_2$ of a normalized pair of solutions of~\eqref{1}
has the property 
\begin{equation}\label{defE}
E(z)=0\Rightarrow E'(z)\in\{\pm1\}.
\end{equation}
Every entire function satisfying~\eqref{defE}
is the product of a normalized pair of solutions of~\eqref{1}, with $A$ given by
\begin{equation}\label{AE}
4A=-2\frac{E''}{E}+\left(\frac{E'}{E}\right)^2-\frac{1}{E^2} .
\end{equation}
Conversely, every entire function $E$ satisfying~\eqref{AE} for some entire 
function $A$ satisfies~\eqref{defE}.
The functions $E$ and $F$ are related by the formula
\begin{equation}
\label{EF}
E=\frac{F}{F'}.
\end{equation}
Note that zeros of $E$ correspond to zeros and poles of~$F$.

We conclude that studying  zeros of linearly independent solutions 
of~\eqref{1} is essentially equivalent to investigating the zeros of
functions satisfying~\eqref{defE}, or the zeros and poles of locally univalent functions.
The relation between the coefficient $A$ in~\eqref{1},
the function $E$ satisfying \eqref{defE}
and the locally univalent function $F$ 
is given by~\eqref{Sch}, \eqref{AE} and~\eqref{EF}.

Functions satisfying~\eqref{defE} play an important role in the 
work of Steven Bank and Ilpo Laine \cite{BL1,BL2}, and they are now called
\emph{Bank-Laine functions}, or BL functions for short. Since this work, 
there has been a substantial interest in BL functions of finite order.
We refer to the surveys \cite{G} and \cite{LT} which cover the literature
before 2008, and to the introductions of the recent papers~\cite{BE1,BE2,L1,L2}.
In particular, much attention has been paid to the exponent of convergence $\lambda(E)$
and the order $\rho(E)$ of a BL function~$E$;
see \cite[p.~7]{Laine1993} for the definitions.

When $A$ is transcendental, all solutions $w$ of \eqref{1} have infinite order.
However, it is possible that the product $E$ of two solutions has finite order. For example, when
\[
A=p''-(p')^2-e^{4p}
\]
with a polynomial~$p$, we have a normalized pair of solutions
\[
w_{1,2}(z)=\frac{1}{\sqrt{2}}\exp\!\left(- p(z)\mp \int_0^z e^{2p(t)}\dif t\right) .
\]
The order of $A$ and of the corresponding BL function $E=w_1 w_2=\exp(-2p)/2$ is the degree of~$p$
and hence an integer.

Bank and Laine conjectured that if $\lambda(E)<\infty$, then
$\rho(A)\in\N$, as is the case in this elementary example.
Langley~\cite{Langley1998} constructed
\emph{non-elementary}
examples of BL functions and corresponding coefficients $A$ of finite integer
order in 1998; further examples were constructed in~\cite{DL,Langley2001}. Recently,
examples of BL functions of any order in $[1,\infty)$ and corresponding coefficients $A$
of any order in $(1/2,\infty)$ were constructed in~\cite{BE1,BE2}, resolving the
conjecture of Bank and Laine in the negative. Note that the order of a
transcendental BL function is at least 1; 
see~\cite{S}, \cite[Corollary~1]{R} and~\cite[Theorem~1]{Shen1987}.

In the recent papers \cite{L1} and \cite{L2}, Jim Langley started to investigate
real BL functions $E$ of finite order for which all zeros are real.
So in this case the associated differential equation~\eqref{1} has two linearly
independent solutions with only real zeros.
As already mentioned, for a polynomial $A$ this is possible only when $A$
is constant \cite[Theorem~3]{HSW}.
In fact, if $A$ has degree~$n$, then the exponent of convergence of the 
non-real zeros of the product of two linearly independent solutions is equal to $(n+2)/2$
\cite[Theorem~1]{Gundersen1986}.
So it is surprising that there exist non-elementary 
BL functions of finite order with only real zeros.

Our first result says that instead of assuming that $E$ has finite order,
it is enough to assume that $A$ has finite order.
Note that we always have $\rho(A)\leq\rho(E)$ by~\eqref{AE}.
\begin{theorem}\label{theorem0}
Let $A$ and $E$ be entire functions satisfying~\eqref{AE}. Suppose that 
the zeros of $E$ lie on finitely many rays emanating from the origin.
Then $\rho(E)<\infty$ if and only if $\rho(A)<\infty$.
\end{theorem}
Theorem~\ref{thm:corollary1} below shows that
we actually have $\rho(A)=\rho(E)$ if the zeros of $E$ are real.

Theorem~\ref{theorem0} yields that
studying entire coefficients $A$ of finite order for which the 
differential equation~\eqref{1} has two linearly independent solutions 
with only real zeros is equivalent to studying 
BL functions of finite order with real zeros.

A meromorphic function $f$
is called {\em real} if it maps $\R$ into $\R\cup\{\infty\}$.
This is equivalent to $f(\overline{z})=\overline{f(z)}$ for all $z\in\C$.
Functions (not necessarily analytic) which satisfy this last equality
will be called {\em symmetric}. For other objects,
like subsets of the plane, the word {\em symmetric} will mean
invariant under complex conjugation.

We say that an infinite real sequence without finite accumulation
points is {\em one-sided} if it is bounded from above or below, and {\em two-sided}
otherwise.

The results of Langley \cite{L1,L2} on BL functions with real zeros
can be summarized as follows. Recall that $\lambda(E)$ denotes the 
\emph{exponent of convergence} of the zeros of $E$.

\begin{thmx}\label{thmA}
Let $E$ be a real Bank-Laine function of finite order with only real zeros and let
$A$ be given by~\eqref{AE}. 
\begin{enumerate}
\item[$(a)$] 
If the zeros of $E$ form an infinite one-sided sequence,
then $\lambda(E)\geq 3/2$.
Moreover, if $\lambda(E)=3/2$, then $\rho(E)=\rho(A)=3/2$.
\item[$(b)$] 
If the zeros of $E$ form an infinite two-sided sequence,
then either $A$ is constant, or $A$ is transcendental and
$\lambda(E)\geq 3$.
Moreover, if $\lambda(E)=3$, then $\rho(E)=\rho(A)=3$.
\end{enumerate}
\end{thmx}
Langley's proofs actually yield a more general result, stated as Theorem~\ref{thmB}
below.
He also constructed examples for which we have equality in the
estimates of $\lambda(E)$ in~$(a)$ and~$(b)$.
Of course, if $A$ is constant, then the possible forms of $E$ can be determined
explicitly. 

We will strengthen Theorem~\ref{thmA} as follows.
\begin{theorem}\label{thm:corollary1}
Let $E$ and $A$ be as in Theorem~\ref{thmA}.
\begin{enumerate}
\item[$(a)$]
If the zeros of $E$ form an infinite one-sided sequence,
then there exists $n\in\N$ with $n\geq 2$ such that $\lambda(E)=\rho(E)=\rho(A)= n-1/2$.
\item[$(b)$]
If the zeros of $E$ form an infinite two-sided sequence
and $A$ is non-constant,
then there exists $n\in\N$ with $n\geq 2$ such that $\lambda(E)=\rho(E)=\rho(A)= 2n-1$.
\end{enumerate}
\end{theorem}
\begin{remark}\label{remark1a}
For an arbitrary real BL function $E$ with only real zeros,
not necessarily of finite order, there are
no restrictions on the exponent of convergence of $E$. Indeed, a result
of Shen \cite{Shen} says that any set without finite accumulation points
is the zero set of a BL function.
\end{remark}

Theorem~\ref{thm:corollary1} will be a corollary of a more general
result, stated as Theorem~\ref{theorem1} below.
However, we will also give a direct proof of Theorem~\ref{thm:corollary1} in
section~\ref{analytic-proof}.
This proof is analytic in nature, while the proof of the more
general Theorem~\ref{theorem1} is geometric.
The proofs of Theorems~\ref{thm:corollary1}  and~\ref{theorem1} are independent 
of each other.

To state Theorem~\ref{theorem1} and Theorem~\ref{thmB}, we
introduce some terminology. All surfaces in this paper are oriented
and have countable base.
A continuous map of surfaces $F\colon X\to Y$ is called
{\em topologically holomorphic}
if for every point $p\in X$ there are local coordinates at $p$ and at $F(p)$
in which $F$ has the form $z\mapsto z^n,$ where $n$ is a positive integer.
According to Sto\"ilow~\cite{Stoilow1938}, all open discrete maps are topologically
holomorphic.

The points where $n\geq 2$ are called {\em critical points}; their images are called \emph{critical values}. The critical values correspond to the 
{\em algebraic singularities} of the inverse $F^{-1}$.
The function $F$ is a local homeo\-morphism if and only if there are no critical points.

The {\em transcendental singularities} of the inverse are defined as follows;
cf.~\cite{Bergweiler1995}. (There it is assumed that $F$ is meromorphic, but
the definition extends to topologically holomorphic functions without change.)
Let $a\in \bC$. Suppose that $D\mapsto U(D)$ 
associates to every topological disk containing $a$
a connected component $U(D)$ of $F^{-1}(D)$, in such a way
that $U(D_1)\subset U(D_2)$ when $D_1\subset D_2$. (Note that $D\mapsto U(D)$ is determined by its values on any base of neighborhoods of~$a$.)
If $\bigcap_{D} U(D)=\emptyset$, then we say that $D\mapsto U(D)$ is a \emph{transcendental singularity} of~$F^{-1}$
over~$a$. In this case, the sets $U(D)$ are called
\emph{tracts} of $F$ over $a$;
any set containing such a tract is called a
\emph{neighborhood} of this singularity.
So we say that a sequence $(z_n)$ in $\C$ {\em converges} to the singularity
$U$ if for every neighborhood $U(D)$, all but finitely many
members of this sequence belong to $U(D)$. 

If there exists $D$ such that  $F(z)\neq a$ for all $z\in U(D)$ (resp.\ such that $F\colon U(D)\to D\setminus\{a\}$ is a universal covering map), then 
the singularity, and the tract $U(D)$, are called \emph{direct} (resp.\ \emph{logarithmic}).
We note that there can be more than one transcendental singularity
over the same point. The number of transcendental (or direct or logarithmic)
singularities over a point $a$ is just the number of different choices 
$D\mapsto U(D)$. For example, the inverse of 
$F(z)=\exp\exp z$ has infinitely many logarithmic singularities over
both $0$ and $\infty$, one logarithmic singularity over~$1$,
and no other singularities.

We also note that $F^{-1}$ has a transcendental singularity over $a$ if and only if 
$a$ is an asymptotic value of~$F$. This means that there exists a curve $\gamma$
tending to~$\infty$ such that $F(z)\to a$ as $z\to\infty$, $z\in\gamma$.
Each neighborhood $U(D)$ then contains a ``tail'' of this curve~$\gamma$.

Langley's paper in fact contains the following generalization of Theorem~\ref{thmA}.
\begin{thmx}\label{thmB}
Let $E$ be a real Bank-Laine function of finite order with only real zeros and let
$A$ and $F$ be as in~\eqref{AE} and~\eqref{EF}.

If $A$ is non-constant, then
the inverse $F^{-1}$ has infinitely many logarithmic singularities over
$0$ and~$\infty$, but the number $m$ of singularities over points in
$\C^*:=\C\setminus\{0\}$ is finite. Moreover, we have the following:
\begin{enumerate}
\item[$(a)$]
If the zeros of $E$ form an infinite one-sided sequence,
then $A$ is non-constant, $m\geq 2$ and $\lambda(E)\geq m-1/2$.
\item[$(b)$]
If the zeros of $E$ form an infinite two-sided sequence
and $A$ is non-constant, then $m\geq 4$ and $\lambda(E)\geq m-1$.
\end{enumerate}
\end{thmx}
We will see that we actually have equality in these estimates of $\lambda(E)$.

To state Theorem~\ref{theorem1},  we also
recall that an entire function $f$ of order $\rho$ has completely
regular growth in the sense of Levin and Pfluger if there exists
a $2\pi$-periodic function $h_f\colon\R\to\R$, not vanishing identically, such  that
\begin{equation}\label{i2}
\log|f(re^{i\theta})|=h_f(\theta)r^\rho +o(r^\rho)
\end{equation}
as $r\to\infty$, for $re^{i\theta}$ outside a union of disks $\{z\colon |z-a_j|<r_j\}$ such that
\begin{equation}\label{i3}
\sum_{|a_j|\leq r} r_j=o(r)
\end{equation}
as $r\to\infty$. The function $h_f$ is called the \emph{indicator} of~$f$.

Our main result is the following theorem.
\begin{theorem}\label{theorem1}
Let $F\colon\C\to\bC$ be a symmetric local homeo\-morphism
with all zeros and poles real.
Suppose that  the number $m$ of singularities of $F^{-1}$ over points in $\C^*$ is
finite, but that $F^{-1}$ has infinitely many singularities over $0$ or $\infty$.

Then there exists a symmetric homeo\-morphism $\phi\colon \C\to\C$
such that $F_0=F\circ\phi$ is a meromorphic function, so that
$E=F_0/F_0'$ is entire and has the following properties:
\begin{enumerate}
\item[$(i)$] 
If $F$ has only finitely many zeros and poles, then
$m\geq 1$ and $\rho(E)=m$.
\item[$(ii)$] 
If the zeros and poles of $F$ form an infinite one-sided sequence,
then $m\geq 2$ and  $\lambda(E)=\rho(E)=m-1/2$.
\item[$(iii)$] 
If the zeros and poles of $F$ form an infinite two-sided sequence,
then $m$ is even, $m\geq 4$ and $\lambda(E)=\rho(E)=m-1.$
\item[$(iv)$] 
The functions $E$ and $A$ in~\eqref{AE} have the same order $\rho=\rho(E)$, and
they are of completely regular growth in the sense of Levin--Pfluger.

For $|\theta|\leq\pi$, the indicator of $E$ is given in case $(i)$ by
$h_E(\theta)=c\cos \rho\theta$ with $c\in\R\setminus\{0\}$ while in case $(ii)$
we have $h_E(\theta)=c\sin(\rho|\theta|)$ with $c>0$ if the zeros are positive
and $h_E(\theta)=c\sin(\rho|\theta-\pi|)$ with $c>0$ if the zeros are negative.
In case $(iii)$, $h_E$ is given by the (now coinciding) formulae of case $(ii)$.
In all cases we have
\begin{equation}\label{h_A}
h_A=2 \max\{-h_E,0\}.
\end{equation}
with some $c>0$.
\item[$(v)$] 
All values of $m$ indicated in $(i)-(iii)$ can actually occur.
\end{enumerate}
\end{theorem}
\begin{remark}\label{remark2}
All assumptions of Theorem~\ref{theorem1}
are of purely topological nature. So Theorem~\ref{theorem1} contains
a parabolic type criterion for a class of surfaces spread over the
sphere. It can be compared with the theorem of Nevanlinna~\cite{Nevanlinna1932} 
describing the conformal type and asymptotic behavior of
a locally univalent function $F$ whose inverse has only finitely many singularities.
In particular, Nevanlinna showed that a meromorphic function $F$ with this property
has finite order.

Suppose that $F^{-1}$ has only 
finitely many singularities over $0$ and $\infty$, but that
$F$ otherwise satisfies the hypotheses of Theorem~\ref{theorem1}. Then 
$F^{-1}$ has only finitely
many singularities, so belongs to the class considered by Nevanlinna. 
A result of Hellerstein, Shen and Williamson \cite[Theorem~2]{HSW} says
that if all zeros and poles of a real meromorphic function $F$ of this class are real,
then $F$ is a linear-fractional transformation or of the form
$F(z)=A\tan(az+b)+B$ with real constants $a,b,A,B$.

The reality of zeros and poles of $F$ is an essential assumption
here and in Theorem~\ref{theorem1}:
The results of \cite{BE2} show in particular that there exist
locally univalent meromorphic functions $F$ whose inverses have
only one singularity over $\C^*$, where the order of $E=F/F'$ can take
any preassigned value in $(1,\infty]$. 
\end{remark}

\begin{remark}\label{remark0}
Theorem~\ref{theorem1} is stronger than Theorem~\ref{thm:corollary1} for several reasons.
First, Theorem~\ref{theorem1} does not require the a priori assumption that $E$ has finite order,
but only the assumption that $F^{-1}$ has finitely many singularities over points in
$\C^*$.
Second, we obtain a more precise
description of the asymptotics of the functions $E$ and~$A$,
namely that these functions are of completely regular growth.
In particular, the functions are of normal type of the given order,
a conclusion that does not follow from our analytic proof of
Theorem~\ref{thm:corollary1}; see Remark~\ref{remark33} below.
Our proof of Theorem~\ref{theorem1} will in fact give additional insights in the 
structure of these functions.
Finally, the geometric approach also allows us to construct examples showing
that all indicated values of $m$ may actually occur.
Note that Langley's
Theorem A gives such examples for $m=2$ in case $(ii)$ and for
$m=4$ in case $(iii)$. Some of the underlying ideas of the construction of our examples
for general $m$ are similar to his, but the details are quite different.
It is plausible that
Langley's methods could also be modified to yield examples for arbitrary~$m$.
\end{remark}

\begin{remark}\label{remark1}
The case of arbitrary (not necessarily real) BL functions $E$ of finite order with 
all zeros real can be reduced to the case of real BL functions by the
following remark of Langley~\cite[p.~228]{L1}: Write $E=\Pi e^{P+iQ}$,
where $\Pi$ is a canonical product with real zeros, and $P$ and $Q$ are real polynomials.
Then condition~\eqref{defE} implies that at every zero $z$ of $E$ we have
$\Pi'(z)e^{P(z)+iQ(z)}=\pm 1$. Since $\Pi'(z)$ and $P(z)$
are real, we conclude that $Q(z)\in \pi\Z$ for every zero $z$ of~$E$.
So $\Pi e^P$ is a real BL function with all zeros real.  Furthermore, if $E$ is real, then
$A$ is also real by~\eqref{AE}, and $F$ in~\eqref{EF} can be chosen real.
Thus it suffices to consider only real functions $F$, $E$ and~$A$.
\end{remark}

\begin{remark}\label{remark3}
The {\em Speiser class}~$S$ is defined as the set of all meromorphic functions
$F\colon\C\to\bC$ for which there exists a finite subset $A$ of $\bC$ such that
$F\colon \C\setminus F^{-1}(A)\to \bC\setminus A$ is an (unramified) covering.
It plays an important role in value distribution theory~\cite{GO}
and holomorphic dynamics~\cite{B,Eremenko1992,6}.

Langley's Theorem~\ref{thmB} says in particular that 
if $E$ is a BL function of finite order, then
the associated locally univalent function $F$ is in~$S$.

Theorem~\ref{theorem1} gives a description of real locally univalent functions $F$
of class $S$ with only real zeros and poles, for which
the inverses have finitely many logarithmic singularities
over values in $\C^*$, and infinitely many logarithmic
singularities over each $0$ and~$\infty$.
Since class $S$ is much studied, this is of independent interest.
\end{remark}

\begin{remark}\label{remark4}
Our proof of Theorem~\ref{theorem1} uses topological arguments, quasiconformal sur\-gery 
and the Teich\-m\"uller--Wit\-tich--Belinskii theorem.
These methods are frequently used to construct {\em examples} of meromorphic functions.
In this paper, we also use this technique to prove a positive result.
\end{remark}

We have discussed the equation~\eqref{1} under the hypothesis that there are
two solutions with only real zeros.
Our final result addresses the case that there are three solutions with this property.
\begin{theorem}\label{theorem4}
Let $A$ be an entire function and suppose that~\eqref{1}
has three  pairwise linearly independent solutions which have only real zeros.
Then $A$ is constant.
\end{theorem}
\begin{remark}\label{remark9}
Our starting point was~\cite[Theorem~3]{HSW} which says that if A is a
polynomial and if~\eqref{1} has a basis of solutions with only real zeros,
then A is  constant. An extension of this result
to linear differential equations of higher order has been given by
Br\"uggemann~\cite[Theorem~5]{Brueggemann1991}
and Steinmetz~\cite[Corollary~2]{Steinmetz1993}.
It would be of interest to which extent our results
generalize to equations of higher order.
\end{remark}

This paper is organized as follows.
Theorem~\ref{theorem0} is proved in section~\ref{proof-thm0}.
In section~\ref{analytic-proof}, we give a purely analytic proof of Theorem~\ref{thm:corollary1}.
The proof of Theorem~\ref{theorem1} given in the subsequent sections is independent of this.
In section~\ref{section3} we collect the necessary prerequisites
on the pasting-and-gluing techniques and line complexes, to make this paper
self-contained. A reader familiar with this technique may pass
to section~\ref{section6}, where we construct examples (Part $(v)$ of Theorem~\ref{theorem1})
and outline the proof of all other parts.
These parts are then proved in sections~\ref{section7}--\ref{section9}.
Theorem~\ref{theorem4} is proved in section~\ref{proof-thm4}.

\begin{ack}
We thank Jim Langley for helpful comments.
We are also grateful to the referee for a careful reading and valuable suggestions.
\end{ack}

\section{Proof of Theorem~\ref{theorem0}} \label{proof-thm0}
We use the standard notation of Nevanlinna theory as given in~\cite{GO} 
or~\cite{Nevanlinna1953}.
The following result is due to Miles~\cite{Miles1979}.
\begin{lemma}\label{lemma-miles}
Let $f$ be an entire function of infinite order and suppose that the zeros of
$f$ lie on finitely many rays emanating from the origin.
Then there exists a set $L\subset [1,\infty)$ of logarithmic density zero such that 
\begin{equation}\label{6a}
\lim_{\substack{r\to\infty\\ r\notin L}} \frac{N(r,1/f)}{T(r,f)}= 0.
\end{equation}
\end{lemma}
\begin{proof}[Proof of Theorem~\ref{theorem0}]
It follows from~\eqref{AE} and the definition of the proximity 
function $m(r,\cdot)$ that 
\begin{equation}\label{6b}
2 m\!\left(r,\frac{1}{E}\right)
= 
m\!\left(r,\frac{1}{E^2}\right)
\leq 
m(r,A)+ m\!\left(r,\frac{E''}{E}\right)
+ 2 m\!\left(r,\frac{E'}{E}\right) +\LO(1).
\end{equation}
Suppose that $E$ has infinite order.
Lemma~\ref{lemma-miles} and the first fundamental theorem 
yield that there exists a set $L$ of logarithmic density zero such that
\begin{equation}\label{6c}
m\!\left(r,\frac{1}{E}\right) \sim T(r,E)
\quad\text{as}\ r\to\infty,\ r\notin L.
\end{equation}
On the other hand, the lemma on the logarithmic derivative~\cite[Chapter~3, \S~1]{GO},
applied to both $E$ and $E'$, implies
that there exists a set $M$ of finite logarithmic measure such that 
\begin{equation}\label{6d}
 m\!\left(r,\frac{E''}{E}\right) + 2 m\!\left(r,\frac{E'}{E}\right)
=\LO(\log T(r,E))+\LO(\log r) 
\quad\text{as}\ r\to\infty,\ r\notin M.
\end{equation}
Combining the last three equations we conclude that
\begin{equation}\label{6e}
(2-o(1)) T(r,E) \leq T(r,A)
\quad\text{as}\ r\to\infty,\ r\notin L\cup M.
\end{equation}
Since $A$ has finite order by hypothesis, this contradicts the assumption that $E$ 
has infinite order.
\end{proof}

\section{Analytic proof of Theorem~\ref{thm:corollary1}} \label{analytic-proof}
Throughout this section, we consider a Bank-Laine function $E$ and the functions
$A$ and $F$ given
by~\eqref{AE} and~\eqref{EF}. To prove Theorem~\ref{thm:corollary1}, 
 we wish to establish upper and lower bounds on the behavior of $E$ on the
 real axis. We begin by proving an upper bound.
 
\begin{proposition}\label{prop:lemma5}
Let $E$ be a real Bank-Laine function of finite order.
If $E$ has infinitely many positive zeros, then
\begin{equation}\label{2f}
\limsup_{x\to+\infty} \frac{\lvert E(x)\rvert}{x}<\infty.
\end{equation}
If $E$ has infinitely many negative zeros, then 
$\limsup_{x\to-\infty} \lvert E(x)/x\rvert<\infty$.
\end{proposition}
\begin{remark}
Our second, geometric, proof of Theorem~\ref{thm:corollary1} yields the
stronger statement that $E(x)$ itself is bounded as $x\to+\infty$ resp.\ as $x\to-\infty$. 
Moreover, this holds not only for the function $E$, which is the product of two solutions of~\eqref{1}, but for any individual solution 
of having infinitely many zeros. See Remark~\ref{remark11} below.
\end{remark}

To prove Proposition~\ref{prop:lemma5}, we consider the function $G$ defined by 
\begin{equation}\label{2a}
G(z)=\frac{E(z)}{z},
\end{equation}
and relate the singularities of $G^{-1}$ over $\infty$ to the 
singularities of $F^{-1}$ over non-zero finite values. 
Langley~\cite[Proposition 2.1, (C)]{L2} proved that
under the hypotheses of Proposition~\ref{prop:lemma5}
every neighborhood of a transcendental singularity of $F^{-1}$ over a non-zero 
finite value contains a neighborhood of 
a singularity of $G$ over $\infty$. We strengthen this result as follows.
 
\begin{proposition}\label{prop:FGsingularities}
Let $E$ be a Bank-Laine function of finite order, $F$ a locally univalent function
satisfying $E=F/F'$ and $G(z)=E(z)/z$. 

Then there is a bijection between the singularities of $G^{-1}$ over $\infty$ and the 
singularities of $F^{-1}$ over values in $\C^*$, with the following property:
Any sequence of points converging to a singularity of $G^{-1}$ over $\infty$ also
converges to the corresponding singularity of $F^{-1}$. 
\end{proposition}
In order to prove Proposition~\ref{prop:FGsingularities} we will use the following lemma.

\begin{lemma}\label{lem:ESingularities}
 Let $F$ be a meromorphic function and set 
   $E := F/F'$. Then every neighborhood of a
   direct transcendental singularity of $E^{-1}$ over $\infty$ contains an asymptotic path for
   some asymptotic value $a\in\C^*$ of $F$. 
\end{lemma}

To prove the lemma, we use the following result of Huber~\cite{Huber1957}; see also~\cite{Lewis1984}.
\begin{lemma}\label{lemma-huber}
Let $u\colon\C\to[-\infty,\infty)$ be subharmonic and let $\lambda>0$.
Suppose that
\[
\lim_{r\to\infty}\frac{\max_{|z|=r} u(z)}{\log r}=\infty.
\]
Then there exists a path $\gamma$ tending to $\infty$ such that 
\[
\int_{\gamma} e^{-\lambda u(z)}|\dif z|<\infty.
\]
\end{lemma}

\begin{proof}[Proof of Lemma~\ref{lem:ESingularities}]
Let $W$ be a neighborhood of a direct singularity of $E$ over $\infty$; we
may 
choose $W$ as a component
of $\{z\colon E(z)|>K\}$ for some $K>0$. Assuming that $K$ is large, $W$ is a direct tract.
This implies that the function 
\[
u(z)=
\begin{cases}
\displaystyle\log\!\left|\frac{E(z)}{K}\right| & \text{if}\ z\in W, \\
0            & \text{if}\ z\notin W,
\end{cases}
\]
satisfies the hypothesis of Lemma~\ref{lemma-huber}; see~\cite[Theorem~2.1]{bergweilerripponstallard}. We apply Lemma~\ref{lemma-huber} with 
$\lambda=1$. It follows that $W$ contains a curve $\gamma$ tending to $\infty$ such that 
\begin{equation}\label{2d}
\int_\gamma \left|\frac{F'(z)}{F(z)}\right|\cdot|\dif z| =
\int_\gamma \frac{|\dif z|}{|E(z)|}<\infty.
\end{equation}
This means that the image of $\gamma$ under a branch of $\log F$ has
 finite Euclidean length, and hence this branch tends to some value $\beta\in\C$ as 
 $z$ tends to $\infty$ along $\gamma$.
 Setting $\alpha:=e^\beta\in\C^*$ we thus have $F(z)\to\alpha$ as $z\to\infty$, $z\in\gamma$. 
\end{proof}

When $F$ is locally univalent, then $E$ is entire and hence every transcendental singularity of 
 $E$ over $\infty$ is direct.  
  In particular, 
  every neighborhood of a transcendental singularity of $G$ over $\infty$ is also
 a neighborhood of a direct singularity of $E$, and thus contains an asymptotic path of $F^{-1}$. 
 Under the hypotheses of Proposition~\ref{prop:FGsingularities}, the corresponding
 singularity of $F^{-1}$ is logarithmic by~\cite[Corollary~1.1]{langley2019sing}.
To prove Proposition~\ref{prop:FGsingularities}, we use an estimate on the derivative 
   of a function having a logarithmic singularity, which is a consequence of
   Koebe's theorem. Such estimates are useful in other
   contexts, notably in the study of the class $B$ in complex dynamics,
   and therefore we state the result in this generality for future reference. 
   (A similar estimate is also used by Langley; see the second displayed formula in the proof of \cite[Proposition 2.1]{L2}.)
\begin{lemma}\label{lem:Koebe}
Let $H = \{z\colon \re z > 0\}$ be the right half-plane and $\phi\colon H\to \C$
be univalent.  Let $z_0,z\in H$ with $\re z \geq \re z_0$. Then
\begin{equation}\label{eqn:phiestimate}
\lvert \phi'(z)\rvert \geq \lvert \phi'(z_0)\rvert \cdot
\left(1 + \frac{\lvert z - z_0\rvert}{\re z_0}\right)^{-4}.
\end{equation}
\end{lemma}
\begin{proof}
Pre- and post-composing by suitable affine maps, we may assume that $z_0=1$ and
$\phi'(z_0)=1$. 
Put $M(z):=(z-1)/(z+1)$. Then $M$ maps $H$ conformally to the unit disk and
we have $M'(z)=2/(z+1)^2$.
Set $\psi :=  \phi\circ M^{-1}$.
Since $\psi'(0)=\phi'(1)/M'(1)=2$,
Koebe's distortion theorem yields that
\[
| \psi'(w) | \geq 2\frac{1-|w|}{(1+|w|)^3}
\]
for $|w|<1$.
Since $\phi'(z)=\psi'(M(z))M'(z)$ we thus have
\[
|\phi'(z)|
\geq 2\frac{1-|M(z)|}{(1+|M(z)|)^3} \cdot\frac{2}{|z+1|^2}
=4\frac{|z+1|-|z-1|}{(|z+1|+|z-1|)^3}
=16\frac{\re(z-1)+1}{(|z+1|+|z-1|)^4} .
\]
Since $|z+1|\leq |z-1|+2$ we conclude that 
\[
|\phi'(z)|
\geq \frac{1}{(1+|z-1|)^4}
\]
for $\re z\geq 1$.
\end{proof}

\begin{proof}[Proof of Proposition~\ref{prop:FGsingularities}]
 Let $E$, $F$ and $G$ be as in the statement of the Proposition.
   Let $\alpha\in\C^*$ be an asymptotic value of $F$. As already mentioned,
   every singularity $\xi$ of $F$ over $\alpha$ is logarithmic. 
   
 \noindent\emph{Claim.} 
    If $U$ is a sufficiently small neighborhood of $\xi$, 
      then $\lvert G(z)\rvert$ is bounded on $\partial U$. Moreover,
      for large enough $R>0$, the set $\{z\colon \lvert G(z)\rvert > R\}\cap U$
      is unbounded and connected. 

 To prove the claim, observe first that $1/G(z)$ is the derivative of 
   $\zeta\mapsto \log F(\exp(\zeta))$, where $\exp\zeta = z$. 
To study this in more detail, choose $\beta$ with $\exp\beta=\alpha$
and put $D:=D(\beta,\varepsilon)$ for some small $\varepsilon>0$.
Here and in the following $D(\beta,\varepsilon)$ denotes the open disk of radius $\varepsilon$ around~$\beta$.
Let $\Omega$ be the connected component of $F^{-1}(\exp(D))$ that is a neighborhood
   of $\xi$. If $\eps$ is sufficiently small, then 
    $F\colon \Omega \to \exp(D)\setminus \{\alpha\}$ is a universal covering and
    $0\notin \overline{\Omega}$.
   Let $T$ be a connected component of $\exp^{-1}(\Omega)$. 
   Let $\lambda$ be the branch of $(\log F)\circ \exp$ on $T$ that takes values in
   $D\setminus\{\beta\}$; then $\lambda$ is also a universal covering. If $\zeta\in T$ and
        $z = \exp(\zeta)$, then 
        \begin{equation}\label{eqn:lambda} \lambda'(\zeta) =  \frac{zF'(z)}{F(z)} = \frac{1}{G(z)}. \end{equation}
        
If $\phi$ is a conformal map from $H$ onto $T$, then
$\lambda\circ\phi\colon H\to D\setminus\{\beta\}$ is a universal covering. 
Another universal covering from $H$ onto $D\setminus\{\beta\}$ is given by 
$w\mapsto \varepsilon \exp(-w) +\beta$. We may normalize $\phi$ so that these 
two maps are equal.
Thus
\[ 
\lambda(\phi(w))=\varepsilon \exp(-w) +\beta 
\] 
for $w\in H$.
Set $g := G\circ \exp\circ \phi$. Then, by~\eqref{eqn:lambda}, 
\begin{equation}\label{eqn:Gphi}
\lvert g(w) \rvert = \lvert G(\exp(\phi(w)))\rvert 
=\frac{1}{|\lambda'(\phi(w))|}
= \frac{1}{\eps}\cdot |\phi'(w)|\cdot \exp(\re w), 
\end{equation}
for $w\in H$. 
Since $T$ is disjoint from its $2\pi i \Z$-translates, we have 
\begin{equation}\label{eqn:ELbound}
\lvert \phi'(w)\rvert \leq \frac{4\pi}{\re w}
\end{equation} 
for all $w\in H$ by Koebe's theorem.
(See~\cite[Lemma~1]{Eremenko1992}, and compare \cite{ELconstant}.) 
Thus, by~\eqref{eqn:Gphi}, $\lvert g(w)\rvert$ is bounded when $\re w$
is bounded away from $0$ and $\infty$.
Note also that if
$w\in H$ with $\re w=R>0$, then $z=\exp\phi(w)$ satisfies 
\[
|(\log F)(z)-\beta|=|\lambda(\phi(w))-\beta|=\varepsilon\exp(-R). 
\]
The first part of the Claim follows for the component $U$ of 
$F^{-1}(\exp(D(\beta, \varepsilon e^{-R})))$ 
contained in~$\Omega$, and in fact for every sufficiently small neighborhood $U$ of~$\xi$.
       
On the other hand, by~\eqref{eqn:Gphi} and Lemma~\ref{lem:Koebe},
for $\re w\geq \re w_0$, 
\begin{equation}\label{eqn:lowerbound}
\frac{\lvert g(w)\rvert}{\lvert g(w_0)\rvert} \geq 
\exp(\re w - \re w_0) \cdot \left(1 + \frac{\lvert w - w_0\rvert}{\re w_0}\right)^{-4} .
\end{equation}
In particular, $g(w)$ tends to infinity along every horizontal line. Moreover, suppose that $\re w_0 \geq 1$ and 
$w$ belongs to the sector of opening angle $\pi/2$ based at $w_0$; i.e., $\re w - \re w_0 > \lvert \im w - \im w_0\rvert$. 
Then the right-hand side of~\eqref{eqn:lowerbound} is bounded below by
$1/(1+\sqrt{2})^4$.
        
Now let $R > 4\pi e$ and let $V$ be a connected component of 
\[
\{w\in H\colon \re w > 1\text{ and } \lvert g(w)\rvert > R\}.
\]
Then $\re w > 1$ for all $w\in\overline{V}$ by~\eqref{eqn:Gphi} 
and~\eqref{eqn:ELbound}, and hence $\lvert g(w)\rvert = R$ for all
$w\in\partial V$. Since $g$ is unbounded on $V$  there is some $w_0\in V$ with
$\lvert g(w_0)\rvert > (1+\sqrt{2})^4 R$. So $V$ contains a sector based at $w_0$ as above, and in particular
all sufficiently large points at argument between $-\pi/5$ and $\pi/5$. Hence the component
$V$ is unique, and the Claim is proved.

  To complete the proof, recall that, since $E$ and hence $G$ are of finite order, the number $n$ of
    singularities of $G^{-1}$ over $\infty$ is finite. Let $K_0>0$ be so large that $\{z\colon \lvert G(z)\rvert > K_0\}$ has exactly $n$ unbounded components,
    one component $V(S)$ for each singularity $S$ of~$G^{-1}$. By Lemma~\ref{lem:ESingularities}, there is a singularity 
    $S'$ of $F^{-1}$ over some value $\alpha\in\C^*$ such that every neighborhood of $S'$ intersects $V(S)$. By the Claim, every neighborhood of $S'$ is also
    a neighborhood of $S$, and hence any sequence of points converging to $S$ also converges to~$S'$. That the map $S\mapsto S'$ is a bijection
    also follows from the Claim. 
\end{proof}

\begin{proof}[Proof of Proposition~\ref{prop:lemma5}]
Suppose that $E$ has infinitely many positive zeros, and let $\Omega$ be a logarithmic tract of $F$
 over a point in $\C^*$. We may assume that $\Omega$ contains no zeros and poles of $F$ 
 and hence no zeros of~$E$. Since a logarithmic tract is simply connected 
and since  a logarithmic tract intersecting the real axis is symmetric,
the intersection of $\Omega$ with the real axis is connected, and hence bounded from above.
So if $(x_n)$ is a sequence tending
to $+\infty$, then $(x_n)$ does not converge to a transcendental singularity of
$F^{-1}$. Hence, by Proposition~\ref{prop:FGsingularities}, a sequence $(x_n)$ tending to $+\infty$
cannot converge to a transcendental singularity of~$G^{-1}$. Thus $|G(x_n)|$
is bounded for such a sequence. In other words,
$\lvert G\rvert$ is bounded on the positive real axis, as claimed. 

The case that $E$ has infinitely many negative zeros reduces to the case
of positive zeros by considering $E(-z)$ instead of $E(z)$.
\end{proof}

Having established an upper bound for $G$ and hence $E$ on the real axis,
we now prove a lower bound, outside certain neighborhoods of the zeros.
We use the following lemma due to Laguerre and Borel; see~\cite{Marden1968}.
\begin{lemma}\label{lemma6}
Let $f$ be a real entire function of finite genus $p$ with $m$ non-real zeros.
Then, in addition to one zero of the derivative $f'$ of $f$ between each pair of adjacent
real zeros of~$f$, the derivative $f'$ has at most $p+m$ real and non-real zeros.
\end{lemma}
Here zeros are counted with multiplicities. The result implies that a real  entire function of finite order with
only finitely many non-real zeros has only finitely many 
local minima where $f(x)\geq 0$ and only finitely many local maxima where $f(x)\leq 0$. Moreover, the same is true 
for all its derivatives. Thus we can apply the following fact to the restriction 
of $E$ to the interval between any two successive (and sufficiently large) zeros.

\begin{lemma}\label{lem:C2}
 Let $f\colon [a,b]\to\R$ be $C^2$ with the following properties. 
 \begin{enumerate}
   \item[$(a)$]  $f(a)=f(b)=0$.
   \item[$(b)$]  $\lvert f'(a)\rvert = \lvert f'(b)\rvert \geq  1$.
   \item[$(c)$] $f'$ has a unique zero $c$ in $(a,b)$.
   \item[$(d)$] $f''$ has at most one zero in $[a,c]$ and at most one in $[c,b]$.
   \item[$(e)$] $f''$ is negative at every local minimum and positive at every
       local maximum. 
 \end{enumerate}
Then
 \begin{equation} \label{x3}
    \lvert f(x)\rvert > \frac{\min\{x-a,b-x\}}{20}
 \end{equation}
 for all $x\in (a,b)$. 
\end{lemma}
\begin{proof}
  By considering the map $x\mapsto f((b-a)x + a)/(b-a)$, we may assume that $a=0$ and $b=1$. 
  Replacing $f$ with $-f$ if necessary, we further assume that $f(x)>0$ for $0<x<1$. Then
 $f'(0)\geq 1$, $f'(1)\leq -1$ and $f''(c)\leq 0$. 

We first claim that 
\begin{equation}\label{eqn:f_lowerbound}
   f(c) = \max_{0\leq x \leq 1} f(x) > \frac{1}{20}. 
\end{equation}
 Suppose, by contradiction, that $f(c) \leq 1/20$. Choose $\eta \in [0,1/3]$ by such that $f''(\eta)$ is minimal. Then
\[ 
f'(x) = f'(0) + \int_{0}^x f''(t)\dif t \geq 1 + \int_0^x f''(\eta) \dif t
=  1 +  f''(\eta) x
\]
for $0\leq x\leq 1/3$, and hence 
\[
\frac{1}{20} \geq f(c) \geq f\!\left(\frac{1}{3}\right) 
= \int_{0}^{1/3} f'(t)\dif t \geq  \int_0^{1/3} (1 + f''(\eta)t)\dif t 
= \frac{1}{3} + \frac{f''(\eta)}{18}.
\]
    Thus $f''(\eta)\leq 18/20 - 6 < -5$. 
    Applying the same argument to $x\mapsto f(1-x)$, we also find $\eta^*\in [2/3,1]$ with $f''(\eta^*)<-5$.

  Without loss of generality,
    we may assume that $c\geq 1/2$. Otherwise, replace $f$ by $x\mapsto f(1-x)$. 
    Let $\tau\in [1/3,1/2]$ be such that $f''(\tau)$ is maximal. For
      $x\in [1/3,1/2]$, we have 
       \[  f'(x) \geq f'(x) - f'\!\left(\frac{1}{2}\right) = -\int_{x}^{1/2} f''(t)\dif t \geq \left(\frac{1}{2}-x\right)\cdot (-f''(\tau)). \]
It follows that
\begin{align*} 
\frac{1}{20} &\geq f(c) 
\geq	 f\!\left(\frac12\right) > f\!\left(\frac12\right) - f\!\left(\frac13\right) 
\\ &
= \int_{1/3}^{1/2} f'(t)\dif t 
\geq -f''(\tau) \int_{1/3}^{1/2} \left(\frac12-t\right)\dif t 
= -\frac{f''(\tau)}{72}, 
\end{align*}
   and thus $f''(\tau)\geq -72/20 > -5>\max\{f''(\eta),f''(\eta^*)\}$. 
   
   It follows that $f''(\tau)$ takes a local maximum between $1/3$ and $2/3$.
By assumption~$(e)$,
     this maximum value is positive. So the interval $[\eta,\eta^*]$ contains at least two zeros of~$f''$,
     bounding an interval on which $f''$ is positive. Since $f''(c)\leq 0$, these zeros either both belong to
     $[0,c]$ or both to $[c,1]$, contradicting $(d)$. Thus~\eqref{eqn:f_lowerbound} is proved. 
     
   To complete the proof of the lemma, we show that 
     \begin{equation}\label{eqn:f_upperbound} \min_{\eps\leq x \leq 1-\eps} f(x) \geq \frac{\eps}{20}. \end{equation}
     The minimum on the left-hand side is assumed either at $x=\eps$ or $x=1-\eps$; we may assume the former. In particular, $c\geq \eps$.
     By the mean value theorem, there exists $\xi\in (0,\eps)$ such that
\begin{equation}\label{x1}
f'(\xi) = \frac{f(\eps)}{\eps}.
\end{equation}
There also exists $\xi^*\in (\eps,c)$ such that
\begin{equation}\label{x2}
f'(\xi^*) = \frac{f(c) - f(\eps)}{c-\eps} \geq \frac{f(c) - f(\eps)}{1-\eps}.
\end{equation}
We may assume  that $f(\eps)\leq \eps$ since otherwise there is nothing to prove.
Then $f'(\xi)\leq 1$. 
We show that $f'(\xi^*)\leq f'(\xi)$.
In fact, otherwise $f'$ would have a local minimum between $0$ and $\xi^*$ and,
since $f'(c)=0$, a local maximum between $\xi$ and $c$, contradicting~$(d)$. Thus
$f'(\xi^*)\leq f'(\xi)$. It now follows from~\eqref{x1} and~\eqref{x2} that 
\[
\eps \cdot (f(c) - f(\eps)) \leq f(\eps)\cdot (1-\eps).
\] 
Thus $\eps f(c) \leq f(\eps)$.
Combined with~\eqref{eqn:f_lowerbound} this yields~\eqref{eqn:f_upperbound} and
hence the conclusion.
\end{proof}
\begin{remark}
The lower bound for $|f(x)|$ in~\eqref{x3} can certainly be improved,
but it suffices for our purposes.
\end{remark}

Given an entire function~$f$, 
a sequence $(r_k)$ is called a sequence of {\em P\'olya peaks of
order $\sigma\in [0,\infty)$} for $\log M(r,f)$, where $M(r,f)=\max_{|z|=r}|f(z)|$ is the maximum modulus,
if for every $\varepsilon>0$ we have
\begin{equation}\label{3a}
\log M(tr_k,f)\leq (1+\varepsilon)t^{\sigma}\log M(r_k,f)
\quad \text{for}\
\varepsilon\leq t\leq \frac{1}{\varepsilon}
\end{equation}
for all large~$k$. Put
\begin{equation}\label{3b}
\rho^*=\sup\left\{ p\in {\mathbb R}  \colon  \limsup_{r,t\to\infty}
\frac{\log M(tr,f)}{t^p\log M(r,f)}=\infty\right\}
\end{equation}
and
\begin{equation}\label{3c}
\rho_*=\inf\left\{ p \in {\mathbb R}  \colon  \liminf_{r,t\to\infty}
\frac{\log M(tr,f)}{t^p\log M(r,f)}=0\right\}.
\end{equation}
Drasin and Shea \cite{Drasin1972} proved that 
P\'olya peaks of order $\sigma$ exist for all finite
$\sigma\in[\rho_*,\rho^*]$ and that we always have
\begin{equation}\label{3d}
0\leq\rho_*\leq\mu(f)\leq \rho(f)\leq\rho^*\leq\infty,
\end{equation}
where $\mu(f)$ denotes the lower order of~$f$.

\begin{proof}[Proof of Theorem~\ref{thm:corollary1}]
We will first consider the case that $E$ has infinitely many positive and 
infinitely many negative zeros; that is, we will prove conclusion~$(b)$
 of the theorem. The minor modifications to handle conclusion~$(a)$
will be discussed at the end of the proof.

We will show that if $(r_k)$ is a sequence of P\'olya peaks of some order~$\sigma$
for $\log M(r,E)$, then $\sigma=N$ for some odd integer~$N$.
In view of~\eqref{3d} this yields that $\mu(E)=\rho(E)=N$.

We consider the subharmonic functions $u_k$ given by
\begin{equation}\label{3e}
u_k(z)= \frac{\log|E(r_k z)|}{\log M(r_k, E)} .
\end{equation}
Given $\varepsilon>0$ we then have 
\begin{equation}\label{3f}
u_k(z)\leq (1+\varepsilon)|z|^{\sigma}
\quad \text{for}\
\varepsilon\leq |z|\leq \frac{1}{\varepsilon} 
\end{equation}
and large~$k$.
This implies (cf.\ \cite[Theorems 4.1.8 and 4.1.9]{Hoermander1990} or
\cite[Theorems 3.2.12 and 3.2.13]{Hoermander1994}) that,
passing to a subsequence if necessary, the limit
\begin{equation}\label{3g}
   u(z)=\lim_{k\to\infty} \frac{\log|E(r_k z)|}{\log M(r_k, E)}
\end{equation}
exists and is either $-\infty$ or a subharmonic function in~$\C$.
Here the convergence is in the Schwartz space~$\mathscr{D}'$.
This implies that we also have convergence in $L^1_{\mathrm{loc}}$.
There are a number of papers where entire and meromorphic functions 
are studied in terms of a subharmonic $u$ obtained as in~\eqref{3g}; 
see, e.g., \cite{Eremenko1993} for further details.

The function $u$ is harmonic in $\C\setminus\R$ and satisfies
\begin{equation}\label{3h}
u(z)\leq |z|^\sigma
\quad \text{for}\ z\in\C
\end{equation}
as well as 
\begin{equation}\label{3h1}
M(1,u)=1,
\end{equation}
so in particular $u\not\equiv-\infty$.
It follows from Proposition~\ref{prop:lemma5} that 
\begin{equation}\label{3i}
u(x)\leq 0 
\quad \text{for}\ x\in\R.
\end{equation}
We shall show that we also have $u(x)\geq 0$ and thus $u(x)=0$ for $x\in\R$.
In order to do so, suppose that $u(x)<0$ for some $x\in\R$. Then there exist
$\delta>0$, $t>0$ and $\xi\in\R\setminus\{0\}$ such that $u(z)<-\delta$ for
$z\in D(\xi,t)$. 
It follows that (see~\cite[(3.2.6)]{Hoermander1994})
\begin{equation}\label{3j}
|E(z)|<1
\quad \text{for}\ z\in D(r_k \xi, r_k t),
\end{equation}
if $k$ is sufficiently large. Hence 
\begin{equation}\label{3k}
|E'(z)|=\left|\frac{1}{2\pi i}\int_{|\zeta-r_k \xi|=r_k t}
\frac{E(\zeta)}{(\zeta-z)^2}\dif\zeta\right|
\leq \frac{4}{r_k t}<1
\quad \text{for}\ z\in D\!\left(r_k \xi, \frac12r_k t\right)
\end{equation}
provided $k$ is sufficiently large.
By the Bank-Laine property~\eqref{defE},
$E$ has no zeros in $D(r_k \xi, r_k t/2)$.

 Suppose that $\xi>0$, 
 and let $(x_n)$ be the sequence of positive zeros of~$E$. (The case $\xi<0$ is analogous, replacing positive by negative zeros). 
We choose $n$ such that $x_n<r_k \xi<x_{n+1}$.
Then $x_n\leq r_k \xi-r_k t/2$ while $x_{n+1}\geq r_k \xi+r_k t/2$.
Applying Lemma~\ref{lem:C2} to the restriction $E|_{[x_n,x_{n+1}]}$, we see that
  \[ E( r_k \xi) \geq \frac{r_k t}{40} > 1 \]
 for sufficiently large $k$. This contradicts ~\eqref{3j}; thus 
\begin{equation}\label{3m}
u(x)= 0 
\quad \text{for}\ x\in\R.
\end{equation}

Since $u$ is harmonic in the upper half-plane, the reflection principle 
for harmonic functions  yields that 
\begin{equation}\label{3n}
v(z)=
\begin{cases} 
u(z) & \text{if}\ \im z\geq 0,\\
-u(\overline{z}) & \text{if}\ \im z< 0,
\end{cases}
\end{equation}
defines a harmonic function $v\colon\C\to\R$.
Next there exists an entire function $w$ with $\re w(0)=0$ such that $v=\im w$. 
It follows from~\eqref{3h} that $w(0)=0$ and 
\begin{equation}\label{3o}
\im w(z) \leq |z|^\sigma
\quad \text{for}\ z\in\C.
\end{equation}
This implies that $w$ is a polynomial and in fact that $\sigma\in\N$ and 
$w(z)=cz^\sigma$ for some $c\in\C$. Hence $\mu(E)=\rho(E)=\sigma\in\N$.

By~\eqref{3h1} we have $|c|=1$.
Since $u(x)=v(x)=\im w(x)=0$ for $x\in\R$ we have $c\in\R$ and thus $c=\pm 1$.
Hence we have $v(re^{i\theta})=cr^\sigma \sin(\sigma\theta)$
and consequently $u(re^{i\theta})=cr^\sigma \sin(\sigma|\theta|)$ for $|\theta|\leq\pi$.
Since $u$ is subharmonic, $u$ cannot have local maxima.
In particular, there are no local maxima on the positive ray.
This yields that $c=1$. Then, since there are no local maxima on the negative 
ray, we conclude that
$\sigma$ is odd, say $\sigma=2n-1$ with $n\in\N$. Thus
\begin{equation}\label{3p}
u(re^{i\theta})=r^\sigma \sin(\sigma|\theta|)
=r^{2n-1} \sin((2n-1)|\theta|)
\quad \text{for}\ |\theta|\leq\pi .
\end{equation}
It follows from~\eqref{3g} and~\eqref{3p} that 
\begin{equation}\label{3p1}
T(r_k,E)=m(r_k,E)\sim \frac{2n}{(2n-1)\pi}\log M(r_k,E)
\end{equation}
and
\begin{equation}\label{3p2}
m\!\left(r_k,\frac{1}{E}\right)\sim \frac{2n-2}{(2n-1)\pi}\log M(r_k,E)
\end{equation}
so that
\begin{equation}\label{3p3}
N\!\left(r_k,\frac{1}{E}\right)\sim \frac{2}{(2n-1)\pi}\log M(r_k,E)
\end{equation}
as $k\to\infty$. This implies that not only $\mu(E)=\rho(E)=\sigma=2n-1$,
but also that $\lambda(E)=2n-1$. 
Moreover, it follows from~\eqref{AE} and the lemma on the
logarithmic derivative~\cite[Chapter~3, \S~1]{GO} that 
\begin{equation}\label{3p4}
m(r,A)=2 m\!\left(r,\frac{1}{E}\right) +\LO(\log r).
\end{equation}
We conclude that we also have $\rho(A)=2n-1$. 
Theorem~\ref{thmA} yields that $n\geq 2$.
This completes the proof of~$(b)$.

Suppose now that all zeros of $E$ are positive.
We proceed as above, but in this case $u$ is harmonic in $\C\setminus [0,\infty)$.
We now apply the above reasoning to $u^*(z)=u(z^2)$.
Then $u^*$ is harmonic in the upper half-plane and $u^*(x)=0$ for $x\in\R$.
We now find that $2\sigma$ is an odd integer, say $2\sigma=2n-1$ with
$n\in\N$,  and that
\begin{equation}\label{3q}
u^*(re^{i\theta})=r^{2\sigma} \sin(2\sigma|\theta|)
=r^{2n-1} \sin((2n-1)|\theta|)
\quad \text{for}\ |\theta|\leq\pi .
\end{equation}
It follows that
\begin{equation}\label{3r}
u(re^{i\theta})=r^{\sigma} \sin(\sigma\theta)
=r^{n-1/2} \sin\!\left(\!\left(n-\frac12\right)\theta\right)
\quad \text{for}\ 0\leq \theta\leq 2\pi  .
\end{equation}
As before we can now conclude that $\lambda(E)=\rho(E)=\rho(A)= n-1/2$.
\end{proof}

\begin{remark}\label{remark33}
One can derive from this proof further regularity
of the asymptotic behavior of $A$ and $E$. Indeed, since we established
that the possible
values of the order $\sigma$ of the P\'olya peaks form a discrete sequence,
we conclude that 
\[
\rho_*=\rho^*=\mu(f)=\rho(f)>0
\]
 in \eqref{3d};
that is, $A$ and $E$ are of regular growth
in the sense of Valiron.
Furthermore, following the argument from \cite[p. 1210-1211]{Eremenko1993},
one can show that
\begin{equation}\label{3s}
\log M(r,E)=r^{\rho}\ell(r),
\end{equation}
where $\ell$ is a slowly varying function in the sense of Karamata;
that is, $\ell(cr)/\ell(r)\to 1$ as $r\to\infty$ for every $c>1$. Moreover,
\begin{equation}\label{3t}
\log|E(re^{i\theta})|\sim r^\rho\ell(r)h(\theta)\quad \text{as}\ r\to\infty,
\end{equation}
outside an exceptional set satisfying \eqref{i3}. Similar statements apply
to the function $A$.

However, as we mentioned in Remark \ref{remark0},
the method of this section
does not allow to 
show that $A$ and $E$ are of normal type.
\end{remark}
\begin{remark}\label{remark34}
In the proof of Theorem~\ref{thm:corollary1} we used Theorem~\ref{thmA} to conclude 
that $n\geq 2$. This can also be seen directly. In fact, suppose we have $n=1$
in case $(b)$. Then~\eqref{3t} takes the form
$\log|E(re^{i\theta})|\sim r\ell(r) \sin|\theta|$.

To estimate the logarithmic derivative of $E$ we use the Schwarz integral formula.
It says that if $g$ is holomorphic
in a domain containing the closed disk $\overline{D(a,t)}$, then
\begin{equation}\label{schwarz0}
g(z)=\frac{1}{2\pi i} \int_{|\zeta-a|=t}\frac{\zeta+z}{\zeta-z}\re g(\zeta)\frac{\dif\zeta}{\zeta} +i\im g(a)
\end{equation}
for $z\in D(a,t)$. It follows that
\begin{equation}\label{schwarz1}
g'(z)=\frac{1}{\pi i} \int_{|\zeta-a|=t}\frac{\re g(\zeta)}{(\zeta-z)^2}\dif\zeta
\quad\text{and}\quad
g''(z)=\frac{2}{\pi i} \int_{|\zeta-a|=t}\frac{\re g(\zeta)}{(\zeta-z)^3}\dif\zeta.
\end{equation}
Since the zeros of $E$ are real, 
we may apply this to a branch $g$ of $\log E$ in the upper or lower half-plane.
Noting that $E'/E=g'$ and $E''/E=g''+(g')^2$ we find that given $\varepsilon>0$
there exists $C>0$ such $|E'(z)/E(z)|\leq C\ell(|z|)$
and $|E''(z)/E(z)|\leq C\ell(|z|)^2$
if $\varepsilon\leq |\arg z|\leq \pi-\varepsilon$ and $|z|$ is large enough.
Together with~\eqref{AE} we see that $|A(z)|=\LO(\ell(|z|)^2)$ and hence
$|A(z)|=\LO(|z|^\varepsilon)$ as $|z|\to\infty$,
$\varepsilon\leq |\arg z|\leq \pi-\varepsilon$. 
Since the order of $A$ is finite,
for sufficiently small $\eps>0$ the  Phragm\'en-Lindel\"of
principle~\cite[\S~I.14]{L} shows that the last estimate also holds 
for $|\arg z|\leq\varepsilon$ and $|\arg z-\pi|\leq\varepsilon$.
It follows that $A$ is constant.
A similar argument can be made in case~$(a)$.
\end{remark}
\begin{remark}\label{remark35}
It follows from the proof of Theorem~\ref{thm:corollary1} that the Nevanlinna deficiency $\delta(0,E)$ 
is positive. In fact, \eqref{3p1} and~\eqref{3p2} yield that $\delta(0,E)=(n-1)/n$.
\end{remark}

\section{Preliminaries for the proof of Theorem~\ref{theorem1}} \label{section3}
\subsection{Gluing of elements}\label{topo-hol}
As a general reference for the concepts considered in this section we 
mention~\cite{Eremenko2004}.
We consider connected bordered oriented surfaces~$D$, not necessarily compact.
The border $\partial D$ is equipped with the standard orientation
(such that the interior stays on the left).

An {\em element} $(D,f)$ is a pair, where $D$ is a bordered surface and 
$f\colon D\to\bC$ is a continuous function that is topologically holomorphic on
the interior of $D$ and locally injective near every point of the border $\partial D$. 
Two pairs $(D_1,f_1)$ and
$(D_2,f_2)$ are called {\em equivalent}
if there is an orientation-preserving homeo\-morphism $\phi\colon D_1\to D_2$
such that $f_1=f_2\circ\phi$.
Of course, we write this as $(D_1,f_1)\sim (D_2,f_2)$.

There is a unique conformal structure on $D$
that makes $f$ holomorphic. Equivalence classes are called {\em Riemann
surfaces spread over the sphere}. By the Uniformization Theorem, each
pair $(D,f)$ with open simply connected $D_0$ is equivalent to
a pair $(D_0,f_0)$ where $D$ is 
the whole plane or an open disk and $f_0$ is meromorphic in~$D_0$.

Consider two elements $(D_1,f_1)$ and $(D_2,f_2)$ and suppose that
there are two closed arcs $I_j\subset\partial D_j$ which are mapped by $f_j$ 
onto the same arc homeo\-morphically, but with opposite orientations.
Recall that, by definition of an element, the  $f_j$ have no critical points on $I_j$.
Then there is a homeo\-morphism $\phi\colon I_1\to I_2$
such that $f_1=f_2\circ\phi$
on $I_1$. By gluing $D_1$ and $D_2$ along this homeo\-morphism
we obtain a new element $(D,f)$, where $f=f_j$ on $D_j$. 
(For the  formal definition
of gluing of topological spaces see, for example, \cite[\S~2.5]{Bo}.)
We can apply this procedure when $(D_1,f_1)=(D_2,f_2)$,
and glue an element to itself.

The following result says that this
gluing operation is compatible with the equivalence relation on elements.

\begin{proposition}\label{proposition1}
If $(D_1,f_1)\sim (D_2,f_2)$ and $(D_3,f_3)\sim (D_4,f_4)$,
and we can glue $(D_1,f_1)$ to $(D_3,f_3)$ along some arcs, then $(D_2,f_2)$
can be glued to $(D_4,f_4)$ along the corresponding arcs, and results
of the gluings are equivalent.
\end{proposition}

\begin{proof}
Let $\phi_1\colon D_1\to D_2$ and $\phi_3\colon D_3\to D_4$ be homeo\-morphisms
as in the definition of equivalence and let
$\psi_1\colon I_1\to I_3$, where $I_1\subset
\partial D_1$ and $I_3\subset\partial D_3$, be the gluing homeo\-morphism. 
Thus $f_1=f_2\circ \phi_1$, $f_3=f_4\circ \phi_3$ and
$f_1(x)=f_3(\psi_1(x))$, $x\in I_1$.
We define
\begin{equation}\label{g1}
\psi_2=\phi_3\circ\psi_1\circ\phi_1^{-1}\colon \phi_1(I_1)\to\phi_3(I_3),
\end{equation}
and verify that
\[
f_2=f_4\circ\psi_2\quad\mbox{on}\ \phi_1(I_1).
\]
So $(D_2,f_2)$ and $(D_4,f_4)$ can be glued along the arcs $\phi_1(I_1)$ and $\phi_3(I_3)$. 

Now the homeo\-morphism $\phi$ from the gluing of $D_1$ and $D_2$
to the gluing of $D_3$ and $D_4$
is given by the formula
\[
\phi(z)=
\begin{cases}
\phi_1(z) &\text{if}\ z\in D_1,\\
\phi_3(z) &\text{if}\ z\in D_2.
\end{cases}
\]
One checks using \eqref{g1} that $\psi_1$ and $\psi_2$
match on the arc along which $D_1$ and $D_2$ 
were glued;  that is, we have
\[
\phi_3\circ\psi_1(x)=\psi_2\circ\phi_1(x)\quad\text{for}\  x\in I_1.
\qedhere
\]
\end{proof}

If one glues two simply connected surfaces along a connected arc
one obtains a simply connected surface. If one glues one simply connected
surface to itself along two disjoint arcs, one obtains a doubly connected
surface of genus zero which is homeo\-morphic to a ring in the plane.
Every doubly connected open Riemann surface of genus zero is conformally
equivalent to a round ring $\{z\colon r<|z|<R\}$, where $0\leq r<R\leq+\infty$. 

\subsection{Cell decompositions}\label{cell-dec}
 A {\em cell decomposition} of a bordered surface $D$ is a representation
of $D$ as a locally finite union of disjoint {\em cells} of
dimensions~$0$ (vertices),~$1$ (edges) and~$2$ (faces), so that the
boundary of each cell is a union of cells of smaller dimension.
Here a vertex is a point. An edge
is homeo\-morphic to an open interval and its  closure is homeo\-morphic
to a closed interval. A face is homeo\-morphic to an open disk 
and its closure is homeo\-morphic to a closed disk or closed half-plane.

The edges and vertices of a cell decomposition can be of two types: those
which belong to $\partial D$ we call {\em boundary edges (vertices)}, and those
which belong to the interior of $D$ we call {\em inner edges (vertices)}.

Two cell decompositions are {\em combinatorially equivalent} if there is a
bijection $p$ of the set of cells of one of them onto the set of cells of
the other, respecting the cell dimensions, and such that
$p(\partial c)=\partial p(c)$ for every cell~$c$.
This is equivalent to the existence of a homeo\-morphism of the ambient surfaces
which maps each cell homeo\-morphically. We call such pairs of cell
decompositions simply {\em equivalent}.

\subsection{Labeled cell decompositions}\label{labeled-cell}
 Let us consider a cell decomposition $C$ of $\bC$, with two vertices
which we denote $\times$ and $\circ$. Suppose that 
there are $q\in\N$ edges, with $q\geq 2$, each edge connecting $\times$ to $\circ$,
and $q$ faces. All faces are
digons. Suppose that a finite set $A$ containing one point in the interior
of each face is given. Let $(D,f)$ be an element
and let $L$ be a cell decomposition of $D$ such that $f$ maps
every cell of $L$ into a cell of $C$ of the same dimension,
with the following properties:
\begin{itemize}
\item[$(a)$] 
The restriction of $f$ onto the closure of each edge of $L$
is a homeo\-morphism onto the closure of an edge of~$C$, 
\item[$(b)$] 
For each face $c$ of~$L$, the restriction of $f$ onto
$\overline{c}\setminus\{ f^{-1}(a)\}$ is
a covering of $f(\overline{c})\setminus\{ a\}$, where $\{ a\}=A\cap \overline{f(c)}$.
\end{itemize}

It follows from  $(b)$ that each interior edge of $L$ belongs
to the boundaries of two distinct faces. A boundary edge evidently
belongs to the boundary of one face.

We {\em label} each vertex of $L$ by $\times$ or $\circ$, according to
its image, and we label each face of $L$ by the element of $A$
which is contained in its image cell. 
This set of labels defines the image under $f$ of
each cell of~$L$. Indeed, a boundary edge can be labeled by the label of
the unique face to which it belongs, and an inner edge can be labeled
by the pair of labels of two faces to whose boundaries it belongs.
Since every two faces of $C$ have at most one common boundary edge,
the image of each edge of $L$ is determined by the labels of faces and vertices.

\begin{proposition}\label{proposition2}
Let $C$ and $A$ be as above,
and $(D_1,f_1),(D_2,f_2)$ be two elements with labeled cell decompositions
$L_1,L_2$ of $D_1,D_2$ satisfying conditions $(a)$ and $(b)$. If $L_1$ is
combinatorially equivalent to $L_2$, where the equivalence respects the
labels, then $(D_1,f_1)\sim (D_2,f_2)$.
\end{proposition}
\begin{proof}
We have to define a homeo\-morphism $\phi\colon D_1\to D_2$ such that
$f_1=f_2\circ\phi$. Let $c\mapsto c'$ be the bijection
between cells of $L_1$ and $L_2$. Since this bijection preserves
labels, for each cell $c$ of $L_1$, we have $f_1(c)=f_2(c')$.
This defines a unique homeo\-morphism $\phi$ of the $1$-skeleton of
$L_1$ onto the $1$-skeleton of $L_2$ such that $f_1=f_2\circ\phi$
on the $1$-skeleton. It remains to extend $\phi$ into faces.
Let $\overline{c}$ and $\overline{c}'$ be closures of some corresponding
faces of $L_1$ and $L_2$. Let $a\in A$ be the label of~$c$.
Let $z_0\in\partial c$, and
$w_0=\phi(z_0)\in\partial c'$. Let $z\in c$, and choose an open arc $\gamma$
in $\overline{c}$ from $z_0$ to~$z$ such that $a\not\in f_1(\gamma)$.
There is a unique lift $\gamma_2$ of this curve by $f_2$ starting at $w_0$.
The endpoint $w$ of $\gamma_2$ determines $\phi(z)$. It is evident
that thus defined $\phi$ is the required homeo\-morphism.
\end{proof}

\subsection{Representation of Speiser class functions by line complexes and trees} \label{section4}
The definition of the Speiser class $S$ given in Remark~\ref{remark3} extends
to topologically holomorphic maps. So we say that 
a topologically holomorphic map $f\colon \C\to\bC$ belongs to class $S$
if there is a finite set $A$ such that 
\begin{equation}\label{3}
f\colon \C\setminus f^{-1}(A)\to\bC\setminus A
\end{equation}
is a covering.
For a linear-fractional transformation, this holds with $A=f(\infty)$,
and to avoid trivialities, we assume that $f$ has at least two singular values
(critical or asymptotic values), so the set $A$ contains at least two
points.
For such an $f\in S$,
consider a cell decomposition $C$ of $\bC$ and a finite set $A$ of
its singular values. Let $L_f=f^{-1}(C)$ be the preimage of $C$.
Then $L_f$ is a labeled cell decomposition of $\C$ satisfying the conditions
stated in~\S~\ref{labeled-cell}.

This labeled cell decomposition $L_f$ is called the {\em line complex} of $f$
corresponding to~$C$. It has the following properties:
\begin{itemize}
\item[$(a)$] 
Its $1$-skeleton is a bipartite graph embedded in the plane.
\item[$(b)$] 
The cyclic order of face labels around each $\times$-vertex is the same:
it coincides with the cyclic order of face labels around the $\times$-vertex of~$C$.
The cyclic order of face labels around an $\circ$-vertex is opposite.
\end{itemize}

For fixed $A$ and~$C$, any cell decomposition of $\C$ with
these two properties arises from some topologically holomorphic map of class~$S$.

It follows from Proposition~\ref{proposition2} that any two functions $f$ and $g$
of class $S$ with the same labeled
cell decomposition $L_f=L_g$ (and same $A$ and $C$) are {\em equivalent}:
$f=g\circ\phi$, where $\phi\colon \C\to\C$ is a homeo\-morphism.

For a local homeo\-morphism we can choose $A$ to be the set of
asymptotic values. When $f$ is a local
homeo\-morphism, all faces of $L_f$ are either digons or $\infty$-gons.
The restriction of $f$ on each digon is a homeo\-morphism, while the restriction
of $f$ on an $\infty$-gon is a universal covering of a face of $C$ minus
the element of $A$ which it contains.

For a local homeo\-morphism $f$, one can describe the cell decomposition
$L_f$ by a simpler
object. If for some pair of vertices there is more than one edge
connecting them, replace all these edges by a single edge. So digons disappear. 
We preserve the labels of the remaining faces and the labels
of vertices. The resulting cell decomposition $T=T_f$
has the following properties:
\begin{itemize}
\item[$(i)$] 
Its $1$-skeleton $T$ is a bipartite graph embedded in $\C$.
\item[$(ii)$] 
The cyclic order of face labels around each $\times$-vertex ($\circ$-vertex)
is consistent with (a restriction of) the
cyclic order of face labels around the $\times$-vertex ($\circ$-vertex) of~$C$. 
\end{itemize}

Given any connected bipartite graph embedded in $\C$ with labeled complementary
components satisfying $(ii)$, one can recover the whole line complex $L$
in a unique way (up to equivalence).

In the case that $f$ is a local homeo\-morphism, $T_f$ is a tree.
So equivalence classes of local homeo\-morphisms $f$ of class $S$ are encoded
by trees embedded in $\C$ with labeled vertices and
complementary components satisfying~$(ii)$.

\subsection{Symmetric local homeomorphisms} \label{section5}
We recall that a local homeo\-morphism $f\colon \C\to\bC$ is called {\em symmetric}
if $f(\overline{z})=\overline{f(z)}$ for all $z\in\C$. In this case, the set of
asymptotic values is symmetric, that is, we have $\overline{A}=A$. A cell decomposition
is called symmetric if the complex conjugation 
$z\mapsto\overline{z}$ maps each cell onto a cell of the same
dimension with complex conjugate label. If $f$ and the cell
decomposition $C$ are symmetric, then the line complex $L_f$ and
the tree $T_f$ are symmetric. Conversely, to each symmetric $C$ and $L$ (or $T$)
corresponds a symmetric local homeo\-morphism~$f$.

The complex conjugation acts on the faces of a symmetric cell decomposition.
So each face of a symmetric cell decomposition is either symmetric
or disjoint from the real line.

Unfortunately, for a symmetric set~$A$, a symmetric cell decomposition $C$
may not exist. (It exists if and only if $A$ contains at most $2$ real points.)
There are several methods of dealing with this difficulty in the study of
symmetric functions of class~$S$; see, for example, \cite{EM,EGS}.
In this paper we will use a somewhat different approach, by using symmetric sub-decompositions
of a generally non-symmetric~$C$. 

\subsection{Quasiconformal mappings} \label{qc-maps}
While the first part of the proof of Theorem~\ref{theorem1} given in section~\ref{section8}
will only deal with topologically holomorphic mappings,
the second part given in section~\ref{section9}
will also use quasiregular and quasiconformal mappings.
For the definition and general properties of quasiconformal mappings we 
refer to~\cite{Be,LV}.
We note that quasiregular mappings are called quasiconformal \emph{functions} in~\cite{LV}.

For a region $D$ and a quasiregular map $f\colon D\to\bC$ we use the notation
\[
\mu_f(z)=\frac{f_{\overline{z}}(z)}{f_z(z)},
\quad
K_f(z)=\frac{1+|\mu_f(z)|}{1-|\mu_f(z)|}
\quad\text{and}\quad
K(f)=\sup_{z\in D}|K_f(z)| .
\]
The following result is a consequence of the existence theorem for a quasiconformal mappings
with prescribed dilatation~\cite[\S~V.1]{LV}.
\begin{lemma}\label{lemma-qr}
Let $f\colon\C\to\bC$ be quasiregular. Then there exists a quasiconformal map
$\phi\colon\C\to\C$ such that $f\circ\phi$ is meromorphic.
\end{lemma}
The next result is known as the Teich\-m\"uller-Wit\-tich-Belinskii theorem \cite[\S~V.6]{LV}.
\begin{lemma}\label{lemma-twb}
Let $U$ and $V$ be neighborhoods of $\infty$ and let $\phi\colon U\to V$ be 
a quasiconformal map with $\phi(\infty)=\infty$. Suppose that 
\begin{equation}\label{5a}
\int_U \frac{K_\phi(z)-1}{x^2+y^2} \dif x\,\dif y < \infty.
\end{equation}
Then there exists $c\in\C^*$ such that 
\begin{equation}\label{5b}
\phi(z)\sim c z 
\quad\text{as}\ z\to\infty .
\end{equation}
\end{lemma}
The \emph{logarithmic area} of a measurable subset $A$ of $\C$ is defined by
\begin{equation}\label{5c}
\logarea A =\int_A \frac{\dif x\,\dif y}{|z|^2} .
\end{equation}
Let $\phi\colon U\to V$ be as in Lemma~\ref{lemma-twb}. Then
\begin{equation}\label{5d}
\int_U \frac{K_\phi(z)-1}{x^2+y^2} \dif x\,\dif y \leq (K(\phi)-1)\logarea(\supp(\mu_f)). 
\end{equation}
We conclude that if
\begin{equation}\label{5e}
\logarea(\supp(\mu_f)) < \infty,
\end{equation}
then~\eqref{5a} and hence~\eqref{5b} hold.

The next lemma is easily proved by direct computation.
\begin{lemma}\label{lemma-Aalpha}
Let $A\subset\C$ be measurable and $\alpha>0$.
Suppose that some branch of $z\mapsto z^\alpha$ is injective on $A$ and 
let $A^\alpha$ be the image of $A$ under this branch. Then
\begin{equation}\label{5f}
\logarea(A^\alpha) = \alpha^2 \logarea(A) .
\end{equation}
\end{lemma}
In particular, $\logarea(A^\alpha)$ is finite if and only if $\logarea(A)$ is finite.

\section{Construction of examples, and outline of the proof of Theorem~\ref{theorem1}} \label{section6}
In this section, we will prove statement $(v)$ of Theorem~\ref{theorem1}.
Note that we only need to construct examples of local homeo\-morphisms
satisfying the assumptions of the theorem with any $m$ as described
in $(i)-(iii)$.
In order to conclude that all values of $m$ do actually occur
for meromorphic functions, we will need the first part of the theorem
which says that there is a homeo\-morphism $\phi$ such that $F_0=F\circ\phi$
is a meromorphic function. This part will be proved later.

\subsection{Examples without zeros or with one zero} \label{fin-zer}
For case $(i)$ in Theorem~\ref{theorem1}, functions with the required
properties are easily given. Given $m\in\N$, 
\[
g_m(z)=\exp\!\left(\int_0^z e^{\zeta^m}\dif\zeta\right)
\]
is a real local homeo\-morphism with $m$ singularities over $\C^*$
which has no zeros or poles while
\begin{equation}\label{G_m}
\begin{aligned}
h_m(z)
&=z\exp\!\left(\int_0^z \frac{e^{-t^m}-1}{t} \dif t\right)
=\exp\!\left(\int_0^1 \frac{e^{-t^m}-1}{t} \dif t +\int_1^z \frac{e^{-t^m}}{t} \dif t\right)
\end{aligned}
\end{equation}
is a real local homeo\-morphism with $m$ singularities over $\C^*$
and one zero at the origin. The inverses of the functions $g_m$ and $h_m$,
as well as those of all other functions we construct in this section, clearly have 
infinitely many singularities over $0$ and~$\infty$.

\subsection{Infinite one-sided sequence of zeros and poles} \label{inf-one}
The required function with $m\geq 2$ is constructed by gluing two
elements. 
Our first element is $(D_1,f_1)$
where $D_1=\{z\colon\re z>0\}$ is the right half-plane and $f_1(z)=\tan z$.

The second element is very closely related to the function $h_{k}$ 
introduced in~\eqref{G_m} with $k=2(m-1)$. For $m\geq 2$  we put
\begin{equation}\label{5g}
c_{2(m-1)}=\lim_{x\to\infty} h_{2(m-1)}(x)
=\exp\!\left(\int_0^1 \frac{e^{-t^{2(m-1)}}-1}{t}
\dif t +\int_1^\infty \frac{e^{-t^{2(m-1)}}}{t} \dif t\right) 
\end{equation}
and define
\begin{equation}\label{5h}
f_2(z)= \frac{i}{c_{2(m-1)}}h_{2(m-1)}(-iz) .
\end{equation}
Our second element is $(D_2,f_2)$ where $D_2=\{z\colon\re z<0\}$
is the left half-plane.

Both $f_1$ and $f_2$ map the imaginary axis homeo\-morphically
onto the interval $(i,-i)$,
with the same orientation.
Viewed as the boundary of $D_1$ and $D_2$,
the imaginary axis has opposite orientations.
Thus $(D_1,f_1)$ and $(D_2,f_2)$ can be glued. The 
resulting function $f$ is locally univalent and has
$m$ singularities over~$\C^*$.
Moreover, the gluing can be done symmetrically so that 
the resulting function $f$ is also symmetric, with all zeros
and poles on the positive 
axis.
\subsection{Finitely many zeros and poles} \label{fin-many}
As in~\S~\ref{inf-one}, we consider the function $f_1(z)=\tan z$, define 
$f_2$ by~\eqref{5h} and put $D_2=\{z\colon \re z<0\}$.
However, this time we put 
$D_1=\{z\colon 0<\re z<n\pi\}$,
with $n\in\N$. Finally, we put $D_3=\{z\colon \re z>n\pi\}$,
$f_3(z)=f_2(z-n\pi)$.
Then we can glue $(D_2,f_2)$ and $(D_1,f_1)$ along the imaginary axis and
$(D_3,f_3)$ and $(D_1,f_1)$ along the line $\{z\colon \re z=n\pi\}$.

The resulting function has $f$ has $n$ poles and $n+1$ zeros.
Of course, $1/f$ has $n+1$ poles and $n$ zeros.

In order to construct an example with the same number $n$ of zeros and poles,
one takes $D_1=\{z\colon 0<\re z<(n-1/2)\pi\}$ and glues $1/f_2$ to it 
along the line $\{z\colon \re z=(n-1/2)\pi\}$.

The number of singularities over $\C^*$ in such examples is even.
To obtain an odd number greater than $1$,
we can use functions $f_1$ and $f_3$ with
different values of $m$. It seems that there are no examples
with $m=1$ and a finite number larger than $1$ of zeros and poles.

\subsection{Infinite two-sided sequence of zeros and poles} \label{inf-two}
These examples use the line complex and the associated tree,
as described in section~\ref{section3}.
We will have four asymptotic values in $\C^*$, say $\{ i,-i,2i,-2i\}$.

We consider a symmetric cell decomposition $C$ of $\bC$ with
two vertices $\times$ and $\circ$ on the real line, such that each face contains
exactly one asymptotic value.
Then we consider the tree $T$ embedded in $\C$ as shown in Figure~\ref{fig:tree4}
for $m=4$.
This tree satisfies all conditions stated in \S~\ref{section4}, for the cell decomposition shown in Figure~\ref{fig:treecell}, 
so it defines a local homeo\-morphism  $F$. Observe that every non-real vertex is adjacent to logarithmic tracts both over $0$ and over $\infty$,
which implies that all of the zeros and poles of $F$ are real; hence $F$ satisfies the hypotheses of Theorem~\ref{theorem1}~$(iii)$, for $m=4$.

\begin{figure}[!ht]
\captionsetup{width=.9\textwidth}
\def\svgwidth{\textwidth}
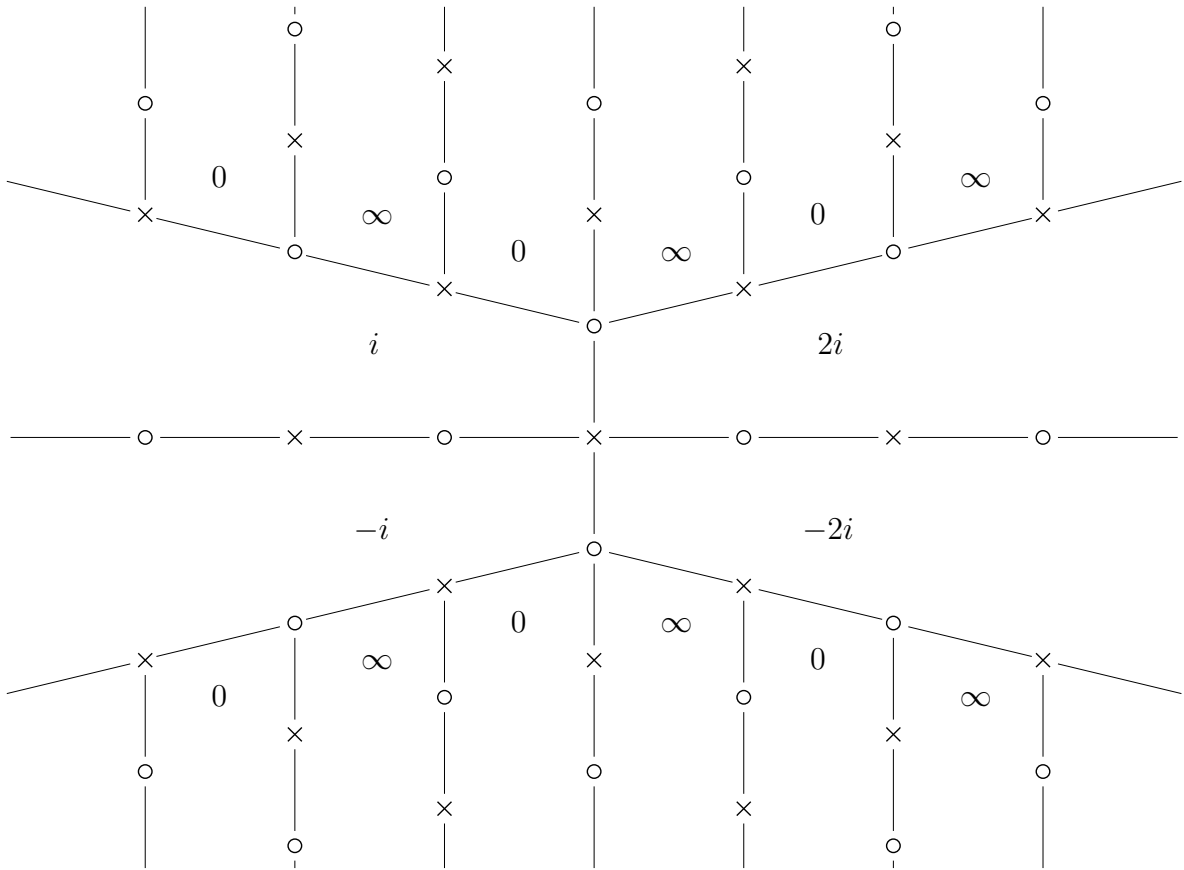
\caption{The tree $T$ for $m=4$.\label{fig:tree4}}
\end{figure}

\begin{figure}[!ht]
\captionsetup{width=.9\textwidth}
\def\svgwidth{\textwidth}
\begingroup%
  \makeatletter%
  \providecommand\color[2][]{%
    \errmessage{(Inkscape) Color is used for the text in Inkscape, but the package 'color.sty' is not loaded}%
    \renewcommand\color[2][]{}%
  }%
  \providecommand\transparent[1]{%
    \errmessage{(Inkscape) Transparency is used (non-zero) for the text in Inkscape, but the package 'transparent.sty' is not loaded}%
    \renewcommand\transparent[1]{}%
  }%
  \providecommand\rotatebox[2]{#2}%
  \newcommand*\fsize{\dimexpr\f@size pt\relax}%
  \newcommand*\lineheight[1]{\fontsize{\fsize}{#1\fsize}\selectfont}%
  \ifx\svgwidth\undefined%
    \setlength{\unitlength}{1200bp}%
    \ifx\svgscale\undefined%
      \relax%
    \else%
      \setlength{\unitlength}{\unitlength * \real{\svgscale}}%
    \fi%
  \else%
    \setlength{\unitlength}{\svgwidth}%
  \fi%
  \global\let\svgwidth\undefined%
  \global\let\svgscale\undefined%
  \makeatother%
  \begin{picture}(1,0.625)%
    \lineheight{1}%
    \setlength\tabcolsep{0pt}%
    \put(0,0){\includegraphics[width=\unitlength,page=1]{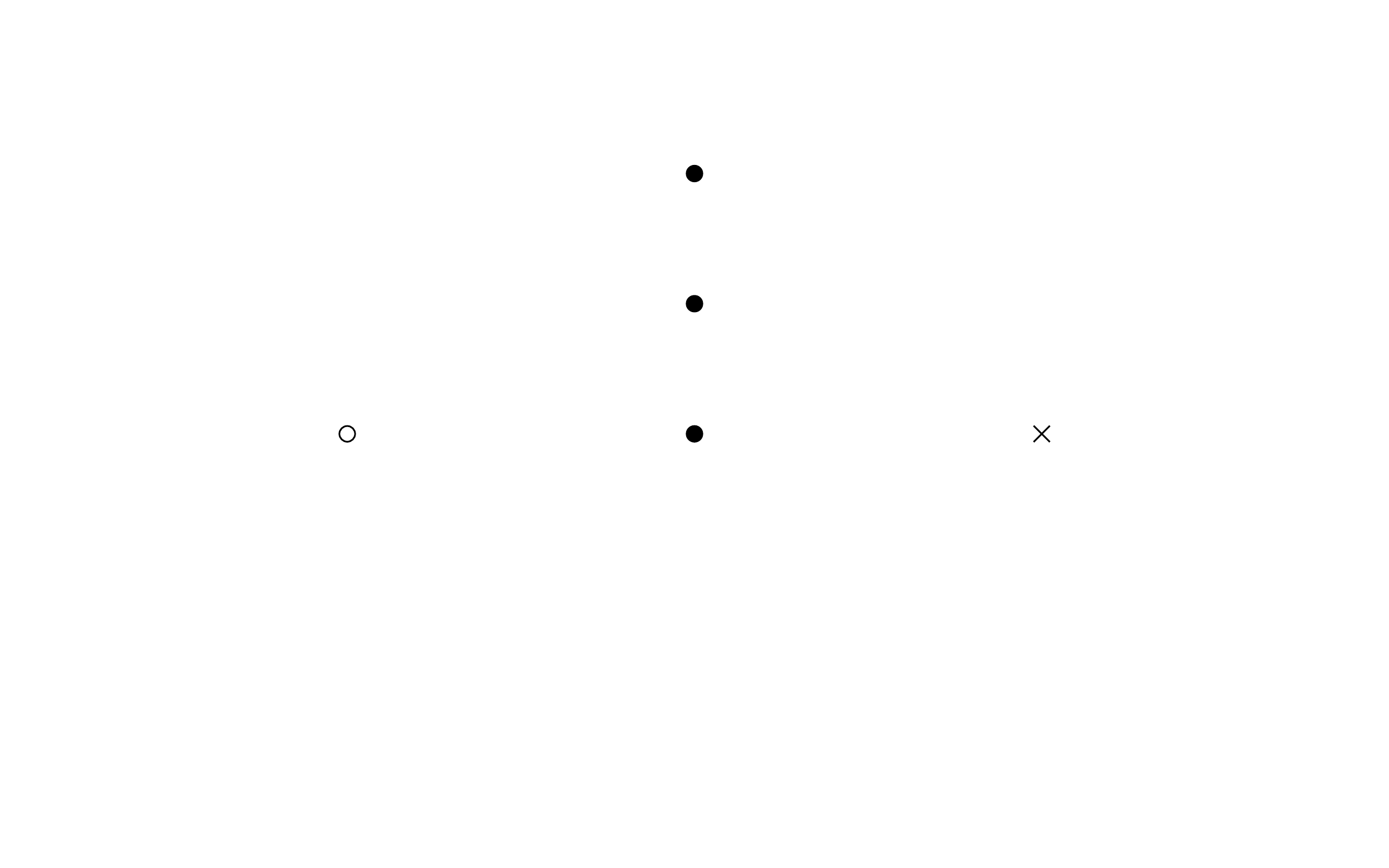}}%
    \put(0.49299316,0.27980471){\color[rgb]{0,0,0}\makebox(0,0)[lt]{\lineheight{1.25}\smash{\begin{tabular}[t]{l}$0$\end{tabular}}}}%
    \put(0.49299316,0.37355469){\color[rgb]{0,0,0}\makebox(0,0)[lt]{\lineheight{1.25}\smash{\begin{tabular}[t]{l}$i$\end{tabular}}}}%
    \put(0.48674317,0.46730469){\color[rgb]{0,0,0}\makebox(0,0)[lt]{\lineheight{1.25}\smash{\begin{tabular}[t]{l}$2i$\end{tabular}}}}%
    \put(0,0){\includegraphics[width=\unitlength,page=2]{fig5.pdf}}%
    \put(0.48674317,0.18605471){\color[rgb]{0,0,0}\makebox(0,0)[lt]{\lineheight{1.25}\smash{\begin{tabular}[t]{l}$-i$\end{tabular}}}}%
    \put(0.48049315,0.09230467){\color[rgb]{0,0,0}\makebox(0,0)[lt]{\lineheight{1.25}\smash{\begin{tabular}[t]{l}$-2i$\end{tabular}}}}%
  \end{picture}%
\endgroup%

\caption{The cell decomposition for the trees in Figures~\ref{fig:tree4} and~\ref{fig:tree6}.\label{fig:treecell}}
\end{figure}

We may insert additional logarithmic tracts by splitting the tree at any vertex that is adjacent only to  a face labeled $0$ and a face labeled $\infty$ (and symmetrically, at its complex
  conjugate), inserting additional branches to ensure that every non-real vertex is again adjacent to a face labeled $0$ and a face labeled $\infty$.  See Figure~\ref{fig:tree6}. Repeating
  this procedure, we obtain local homeo\-morphisms with the desired properties for every even $m\geq 4$, as claimed. 

\begin{figure}[!ht]
\captionsetup{width=.9\textwidth}
\def\svgwidth{\textwidth}
\begingroup%
  \makeatletter%
  \providecommand\color[2][]{%
    \errmessage{(Inkscape) Color is used for the text in Inkscape, but the package 'color.sty' is not loaded}%
    \renewcommand\color[2][]{}%
  }%
  \providecommand\transparent[1]{%
    \errmessage{(Inkscape) Transparency is used (non-zero) for the text in Inkscape, but the package 'transparent.sty' is not loaded}%
    \renewcommand\transparent[1]{}%
  }%
  \providecommand\rotatebox[2]{#2}%
  \newcommand*\fsize{\dimexpr\f@size pt\relax}%
  \newcommand*\lineheight[1]{\fontsize{\fsize}{#1\fsize}\selectfont}%
  \ifx\svgwidth\undefined%
    \setlength{\unitlength}{1200bp}%
    \ifx\svgscale\undefined%
      \relax%
    \else%
      \setlength{\unitlength}{\unitlength * \real{\svgscale}}%
    \fi%
  \else%
    \setlength{\unitlength}{\svgwidth}%
  \fi%
  \global\let\svgwidth\undefined%
  \global\let\svgscale\undefined%
  \makeatother%
  \begin{picture}(1,0.75)%
    \lineheight{1}%
    \setlength\tabcolsep{0pt}%
    \put(0,0){\includegraphics[width=\unitlength,page=1]{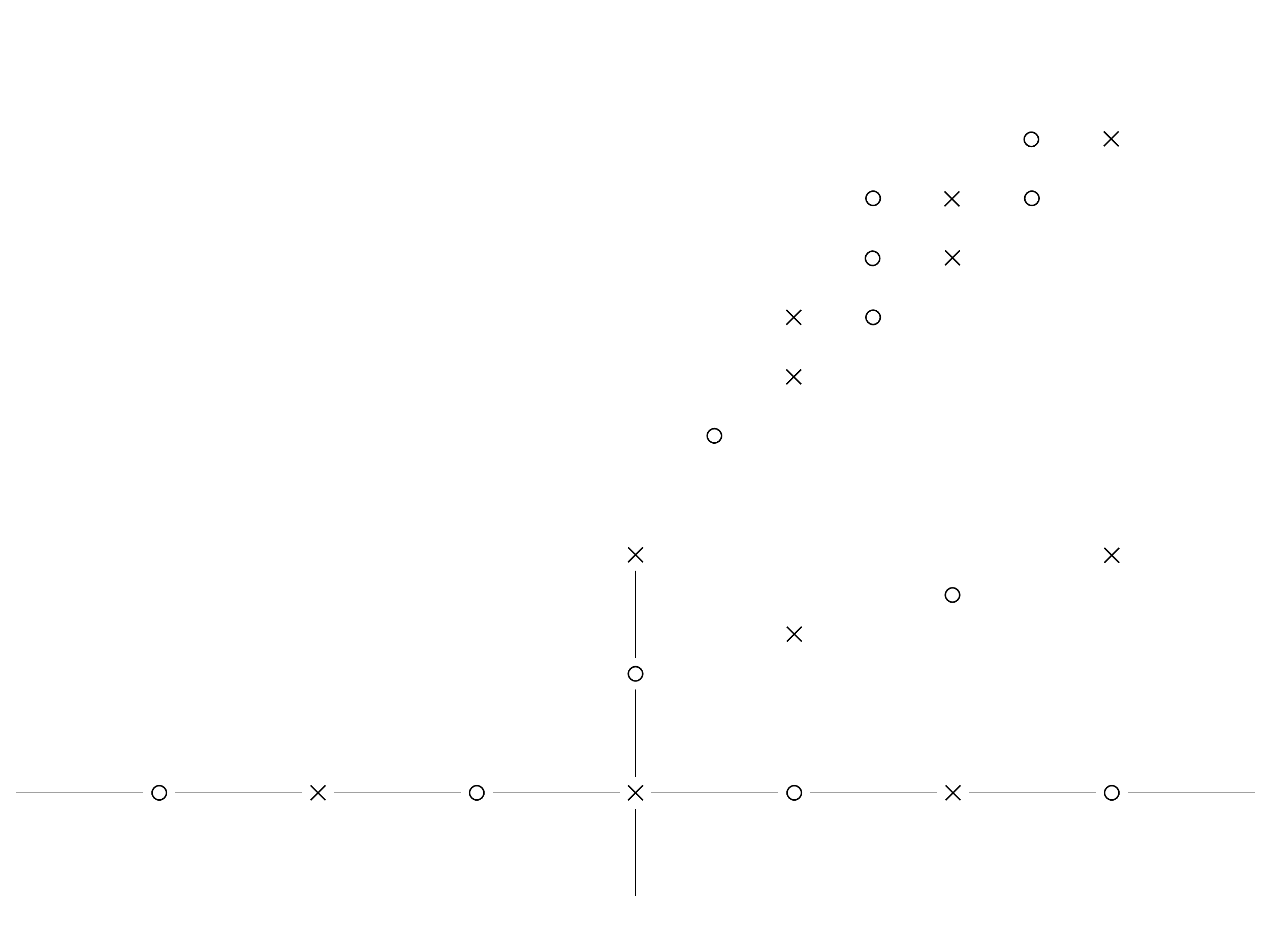}}%
    \put(0.68714954,0.19584911){\color[rgb]{0,0,0}\makebox(0,0)[lt]{\lineheight{1.25}\smash{\begin{tabular}[t]{l}$2i$\end{tabular}}}}%
    \put(0.31214951,0.19584911){\color[rgb]{0,0,0}\makebox(0,0)[lt]{\lineheight{1.25}\smash{\begin{tabular}[t]{l}$i$\end{tabular}}}}%
    \put(0.48674317,0.46105469){\color[rgb]{0,0,0}\makebox(0,0)[lt]{\lineheight{1.25}\smash{\begin{tabular}[t]{l}$i$\end{tabular}}}}%
    \put(0,0){\includegraphics[width=\unitlength,page=2]{fig6.pdf}}%
    \put(0.56174317,0.30480469){\color[rgb]{0,0,0}\makebox(0,0)[lt]{\lineheight{1.25}\smash{\begin{tabular}[t]{l}$\infty$\end{tabular}}}}%
    \put(0.43674316,0.30480469){\color[rgb]{0,0,0}\makebox(0,0)[lt]{\lineheight{1.25}\smash{\begin{tabular}[t]{l}$0$\end{tabular}}}}%
  \end{picture}%
\endgroup%

\caption{The tree for $m=6$. (Since the tree is symmetric, only its upper part is shown. Unlabeled faces are asymptotic tracts over $0$ or~$\infty$.)\label{fig:tree6}}
\end{figure}

\subsection{Outline of the proof of Theorem~\ref{theorem1}} \label{outline}
We will show that all functions $F$
satisfying the conditions of Theorem~\ref{theorem1} are similar to these examples.
To do this, we draw $m$ disjoint asymptotic curves $\gamma_j$ corresponding
to all distinct singularities of $F^{-1}$ over~$\C^*$. These curves will
split the plane into $m$ disjoint sector-like regions $G_0,\ldots,G_{m-1}$
plus a compact set.
Assuming that $F$ has infinitely many positive zeros and poles, we enumerate them
 so that the positive zeros and poles are contained in $G_0$ and that the whole partition of
the plane into $G_0,\ldots,G_{m-1}$ is symmetric. So when there are infinitely many
negative zeros and poles, they will lie in $G_{m/2}$.

Then we will show that the restrictions of our function $F$ onto
$G_j$ are of two special types
described in the next section. One type which we call $\Tan_a$ is symmetric, has
infinitely many zeros and poles,
but $0$ and $\infty$ are not asymptotic values for this element.
This element is similar to the element $(D_1,f_1)$ considered in~\S~\ref{inf-one}.
The second type $\Broom_{d_1,d_2}$  has no zeros and poles,
but infinitely many singularities over $0$ and $\infty$.
It is similar to the element $(D_2,f_2)$ of \S~\ref{inf-one}, with $m=2$.

In the next section we construct explicit
quasiregular representatives of these classes of elements,
and show that their quasiconformal dilatation is supported on a small set
(of finite logarithmic area). This will allow us to paste them together
to reconstruct our function~$F$. All asymptotic properties
will follow from the Teich\-m\"uller--Wit\-tich--Belinskii theorem 
(Lemma~\ref{lemma-twb}).

\section{Elements of special type} \label{section7}
As explained in~\S~\ref{outline}, we will divide the plane into
certain regions $G_j$ such that the restriction of $F$ to $G_j$ is equivalent to one
of two special types of elements.  
These elements $(D,f)$ have the following properties:
The domain  $D$ is bounded by a curve $\gamma\colon\R\to\C$ such
that $|\gamma(t)|\to\infty$ as $t\to\pm\infty$. Here $\gamma$
is oriented in the positive direction so that $D$ is on the left of~$\gamma$.
The function $f$ is such that $f(\gamma(t))$ tends
radially to certain asymptotic
values as $t\to\pm\infty$. This means that there exist
$t_\pm\in \R$, $a_\pm\in \C^*$, $\theta_\pm\in (-\pi,\pi]$ and
$\varepsilon>0$ such that
\begin{itemize}
\item[$(*)$]
$f\circ \gamma$ maps
$[t_+,\infty)$ homeo\-morphically onto $(a_+,a_+ +\varepsilon e^{i\theta_+}]$
and $(-\infty,t_-]$ homeo\-morphically onto $(a_-,a_- +\varepsilon e^{i\theta_-}]$.
\end{itemize}
We will only need the cases where 
$\theta\in\{0,\pm\pi/2,\pi\}$ 
so that $f$ approaches the asymptotic value parallel to the real or imaginary axis.

We call the pairs $d_\pm:=(a_\pm,\theta_\pm)$ \emph{oriented asymptotic values}.
If there exist $t_\pm\in\R$ and $\varepsilon>0$ such that $(*)$ holds, then we
say that the element $(D,f)$ has the oriented asymptotic values $d_\pm$ at $\pm\infty$.
We shall assume that if $a_+=a_-$, then $\theta_+=\theta_-$, since
this suffices for our purposes.

One element will essentially be the restriction of the tangent function
to the right half-plane $H:=\{z\colon \re z>0\}$. The boundary curve $\gamma_H$
is  then given by $\gamma_H(t)=-it$. Then~$(*)$ holds with $a_\pm=\mp i$
and $\theta_\pm=\pm \pi/2$.
Thus $(H,\tan)$ has the oriented asymptotic values $(\mp i,\pm \pi/2)$ at $\pm\infty$.
For technical reasons, we will, however, later introduce some modification of the element $(H,\tan)$.
Essentially, we will show that if $F$ has infinitely many zeros and poles in $G_j$, then $(G_j,F)$ is
equivalent to this (modified) element.

But before doing so we will deal with the domains $G_j$ where $F$ has no
zeros and poles.
\subsection{Elements without zeros and poles}\label{elem-nozeros}
The domains where $F$ maps to $\C^*$ are covered by the following definition.
\begin{definition}\label{def-broom}
Let $D$ be an unbounded simply-connected domain bounded by a curve
$\gamma\colon\R\to\C$, with $\gamma(t)\to\infty$ as $t\to\pm\infty$,
and let $F\colon D\to \C^*$ be a local homeo\-morphism.
Suppose that if $c\in\C^*$, then there exists $\delta>0$ such that
$F^{-1}(D(c,\delta))$  has no unbounded connected component whose closure is in~$D$.

Let $d_\pm = (a_\pm,\theta_\pm)$,
where $a_\pm\in\C^*$ and $\theta_\pm\in(-\pi,\pi]$,
with $\theta_+=\theta_-$ if $a_+=a_-$.
If $(D,F)$ has the oriented asymptotic values $d_\pm$ at $\pm\infty$,
then $(D,F)$ is called of type $\Broom_{d_+,d_-}$.
\end{definition}

We will show that these elements can be represented in a particular form,
similar to the ones considered in~\S~\ref{fin-zer}.
We begin with the case where $a_\pm>0$ and $\theta_\pm=0$.
Here and in the following we put $H^+=\{z\colon \im z>0\}$.
\begin{proposition}\label{lemma-broom1}
Let $a_\pm>0$ and put $d_\pm =(a_\pm,0)$.
Let $(D,F)$ be of type $\Broom_{d_+,d_-}$.
Then there exist compact sets $K$ and $K_0$,
a symmetric rational function $R_0$ with at most one pole and $\xi,c_0 \in\R$
such that with
\begin{equation}\label{10a}
F_0(z)=\exp\!\left( \int_{\xi}^z R_0(t)e^{- t^2}\dif t +c_0 \right)
\end{equation}
we have $(D\setminus K, F)\sim (H^+\setminus K_0, F_0)$.
\end{proposition}
To prove Proposition~\ref{lemma-broom1}, we will use the following lemma.
We omit its proof, but note that it can easily 
be deduced from a result of Morse and Heins~\cite[Theorem~20.4]{M}.
\begin{lemma}\label{lemma-mh}
Let $0<t<T$ and let $f\colon \{z\colon |z|>t\}\to\bC$ be a topologically holomorphic
map. Then there exists a topologically holomorphic map $F\colon\C\to\bC$ such that
$F(z)=f(z)$ for $|z|>T$.

The map $F$ can be chosen to have at most one pole, which is located at~$0$.
Moreover, if $f$ is symmetric, then $F$ can be chosen to be symmetric.
\end{lemma}
The following result is due to Lindel\"of \cite[Chapter~5, Lemma~1.2]{GO}.
\begin{lemma}\label{lemma-lindeloef}
Let $f\colon\{z\colon \re z\geq 0\}\to\C$ be continuous and bounded.
Suppose that $f$ is holomorphic in $\{z\colon \re z> 0\}$ and that there exist
$a_\pm\in\C$ such that $f(it)\to a_\pm$ as $t\to\pm\infty$. Then
$a:=a_+=a_-$ and $f(z)\to a$ as $|z|\to\infty$.
\end{lemma}

We will also use the following result.
\begin{lemma}\label{lemma-QeP}
Let $f$ be a meromorphic function with only finitely many zeros and poles.
Suppose that there exists $N\in\N$ such  that 
for each $R>0$ the set $\{z\colon |f(z)|>R\}$ has at most $N$ unbounded components.
Then $f$ has the form
\begin{equation}\label{10b}
f(z)=Q(z)e^{P(z)}, 
\end{equation}
where $Q$ is rational and where $P$ is a polynomial of degree at most~$N$.
\end{lemma}
\begin{proof}
It is clear that $f$ has the form~\eqref{10b} with a rational function $Q$ and
an entire function~$P$. We have to prove that $P$ is a polynomial
of degree at most~$N$.

Let 
\[
u(z)=\log|f(z)|=\re P(z)+\log|Q(z)|.
\]
For $K>0$ we consider the sets
\begin{equation}\label{10c}
A=\{z\colon u(z)>K\}
\quad\text{and}\quad
B=\{z\colon u(z)<-K\} .
\end{equation}
We may choose $K$ so large that every connected component of $A$ containing one of
the finitely many poles of $f$ is bounded, and such that $A$ has the maximal number
of unbounded connected components subject to this condition.
Then each unbounded connected component of $A$ is a neighbourhood of a
unique singularity of $f^{-1}$ over~$\infty$.
Similarly, for large $K$ 
each unbounded connected component of $B$ is a neighbourhood of a
unique singularity of $f^{-1}$ over~$0$.
By Lindel\"of's lemma~\ref{lemma-lindeloef},
``between'' two such singularities of $f^{-1}$ over $0$  there is a singularity of
$f^{-1}$ over $\infty$, and vice versa. Therefore, between two unbounded components of $A$ 
there is an unbounded component of $B$, and vice versa. In particular, $A$ and $B$ have the same
finite number of unbounded connected components for large $K$.

By \cite{Talpur1976,Lewis1984}
each unbounded component of $A$ contains a path $\gamma$ such that 
$u(z)/\log |z|\to\infty$ as $z\to\infty$, $z\in\gamma$,
while each unbounded component of $B$ contains a path $\gamma$ such that 
$u(z)/\log |z|\to -\infty$ as $z\to\infty$, $z\in\gamma$.
Thus, given $K'>0$, the ``tail'' of such a curve $\gamma$ is contained in a component of 
\begin{equation}\label{10d}
A'=\{z\colon \re P(z)>K'\}
\quad\text{or}\quad
B'=\{z\colon \re P(z)<K'\},
\end{equation}
respectively. (Note that the components of $A'$ and $B'$ are always unbounded.)
By the same argument, the components of $A'$ and $B'$ contain curves whose tails are 
contained in unbounded components of $A$ and $B$, respectively.
Overall we see that there is a one-to-one correspondence
between the unbounded components of $A$ and the components of~$A'$.

We claim that $A'$ has connected complement if $K'$ is sufficiently large.
In order to show this we note that every complementary component of $A'$ must contain
some connected component of $B'$, and thus the number of complementary components is
finite. In particular, there is $R>0$ such that, for sufficiently large $K'$, every complementary
component of $A'$ intersects the disk $D(0,R)$.
If additionally $K' > M(R,P)$, then this disk does not intersect $A'$, and therefore
the complement of $A'$ is connected.
This implies that the complement of every component of $A'$ is connected.

By hypothesis, $A$ has at most $N$ unbounded components, and hence $A'$ has at most $N$ components. 
Let now $w\in\C$ with $\re w>K'$ such that $w$ is not a critical value
of $P$ and let $z_1,z_2\in\C$ be such that $P(z_1)=P(z_2)=w$. 
We will show that if $z_1\neq z_2$, then $z_1$ and $z_2$ are contained
in different components of~$A'$.  Thus $w$ has at most $N$ preimages.
Since this holds for every $w$ with $\re w>K'$ which is not a critical value
of $P$, the conclusion follows.

Thus suppose that $z_1\neq z_2$ and
let $\varphi_1$ and $\varphi_2$ be branches of $P^{-1}$ such that 
$\varphi_1(w)=z_1$ and $\varphi_2(w)=z_2$.
By the Gross star theorem~\cite[p.~292]{Nevanlinna1953}
there exists $t\in (-\pi/2,\pi/2)$ such that 
both $\varphi_1$ and $\varphi_2$ can be continued analytically along
the ray $\{w+re^{it}\colon r\geq 0\}$.
For $j=1,2$, let $\gamma_j$ be the image of this ray under~$\varphi_j$.
Then $\gamma_j$ is a curve connecting $z_j$ with~$\infty$.

Suppose that now that $z_1$ and $z_2$ are in the same component $U$ of~$A'$.
We connect $z_1$ and $z_2$ by a simple curve $\gamma_0$ in $U$ which intersects
$\gamma_1$ and $\gamma_2$ only at $z_1$ and $z_2$. Since $U$ has connected complement, the curves 
$\gamma_0$, $\gamma_1$ and $\gamma_2$ bound a subdomain $V$ of~$U$.
Its image $P(V)$ is an unbounded domain whose boundary is contained
in the union of the ray $\{w+re^{it}\colon r\geq 0\}$ with $P(\gamma_0)$.
This implies that $P(V)$ contains points with real part less than~$K'$.
This is a contradiction since $V\subset U$ and $U$ is a component of~$A'$.
\end{proof}

\begin{proof}[Proof of Proposition~\ref{lemma-broom1}]
Without loss of generality we may assume that $F$ is holomorphic, $D=H^+$ and  $\gamma(t)=t$.
Since $(D,F)$ has the oriented asymptotic values $(a_\pm,0)$ at $\pm\infty$,
there exists a compact interval $I$ such that, by reflection,
$F$ extends to a map holomorphic in $\C\setminus I$.
By Lemma~\ref{lemma-mh} there exist $R>0$ and
a symmetric,  topologically holomorphic map $F_1\colon \C\to\bC$ with no poles in
$\C^*$ such that $F_1(z)=F(z)$
for $|z|>R$. Next, there exists a homeo\-morphism $\phi_1\colon\C\to\C$ such that
$F_0=F_1\circ \phi_1$ is meromorphic in~$\C$.
Here $\phi_1$ can be chosen to be symmetric so that $F_0$ is also symmetric
and has no poles in~$\C^*$.

The function $F_0$ has only finitely many zeros and no poles in~$\C^*$.
Thus there exist a symmetric rational function $Q$ and a symmetric entire function $g$ such that
\[
F_0(z)=Q(z)e^{g(z)}.
\]
Moreover, $F_0$ has only finitely many critical points. Thus
\[
L(z):=\frac{F_0'(z)}{F_0(z)}= g'(z)+\frac{Q'(z)}{Q(z)}
\]
has only finitely many zeros.
Hence there exist a  symmetric rational function $R_0$ with no poles in $\C^*$
and a symmetric entire function $h$ such that
\begin{equation}\label{10z}
L(z)=R_0(z) e^{h(z)}.
\end{equation}
Note that $F_0$ has the form
\begin{equation}\label{F0z}
F_0(z)=\exp\!\left(\int_{1}^z L(t)\dif t+c_0\right)
\end{equation}
for some $c_0\in\R$.
The conclusion thus follows if we can show that $h$ in~\eqref{10z} is a quadratic polynomial,
since once this is known, we can achieve by an affine change of variable that $h(z)=-z^2$.
This will change the lower limit in the integral in~\eqref{F0z} from $1$ to some 
$\xi\in\R$.

For $K>0$, consider a component $U$ of
\begin{equation}\label{10y}
\{z\colon |L(z)|<e^{-K}\} =\{z\colon \re h(z)+ \log |R_0(z)|<-K\}
\end{equation}
which does not contain any zeros of $L$.
Huber's~\cite{Huber1957} Lemma~\ref{lemma-huber} yields that $U$ contains a curve $\Gamma$ tending to $\infty$
such that
\[
\int_{\Gamma} |L(z)|\cdot|\dif z| <\infty.
\]
Hence there exists $\alpha\in\C$ such that
if $z$ is a point on $\Gamma$ and if $\Gamma_z$ denotes the part of $\Gamma$ which connects the starting point of
$\Gamma$ with~$z$, then
\[
\int_{\Gamma_z} L(t)\dif t =
\int_{\Gamma_z} \frac{F_0'(t)}{F_0(t)} \dif t \to \alpha
\quad\text{as}\ z\to\infty,\; z\in\Gamma.
\]
Hence there exists $\beta\in\C^*$ such that $F_0(z)\to\beta$ as $z\to\infty$, $z\in\Gamma$.

By hypothesis, the inverse of $F$ and hence that of $F_0$ have no
singularities over values in $\C^*$ except for
the two asymptotic values $a_\pm$, for which the positive and negative
real axis are asymptotic paths.

This implies (cf.\ Proposition~\ref{prop:FGsingularities})
that the set considered in~\eqref{10y} has at most two components
which do not contain a zero of $L$.
Lemma~\ref{lemma-QeP} implies that $h$ is a polynomial of degree at most $2$.
In fact, since the positive and negative axis are asymptotic paths, the
degree of $h$ is exactly~$2$.
\end{proof}

Proposition~\ref{lemma-broom1} required that $a_\pm>0$ and $\theta_\pm=0$.
We will now lift this restriction.
In our applications, it will be convenient to work with 
the slit plane $\Omega^0:=\C\setminus [0,\infty)$ instead of the upper half-plane~$H^+$.
Note that $z\mapsto z^2$ maps $H^+$ onto~$\Omega^0$.
The boundary of $\Omega^0$ is parametrized by the curve $\gamma_\Omega\colon\R\to\partial\Omega$ which is given by 
$\gamma_\Omega(t)=-t$ for $t<0$ and $\gamma_\Omega(t)=t$ for $t\geq 0$.
So we consider the bordered surface $\Omega$ whose interior
is $\Omega^0$ and the border $\partial\Omega$ is the curve $\gamma_\Omega$.
We put $\Delta=\{z\colon |z|\geq 1\}$.
\begin{proposition}\label{lemma-broom3}
Let  $(D,F)$ be of type $\Broom_{d_+,d_-}$. Then there exist compact sets $K$ and~$K'$,
a local homeo\-morphism 
$B\colon \Omega\setminus K'\to\C$
and $t_0>0$
such that $(\Omega\setminus K', B)\sim (D\setminus K,F)$ and
\begin{equation}\label{10x}
B(\gamma_\Omega(t))-a_\pm = e^{i\theta_\pm}\exp(-|t|)
\quad\text{for}\ |t|\geq t_0.
\end{equation}
The map $B$ is quasiregular with
\begin{equation}\label{l_B}
\logarea\!\left( \supp\!\left(\mu_{B}\right)\cap \Delta \right) <\infty
\end{equation}
and there exists $c\in\C^*$ and $d\in\R$ with $2d\in\Z$ such that,
as $z\to\infty$, 
\begin{equation}\label{12c}
B(z)-a_\pm\sim c z^d \exp(-z)
\end{equation}
 in any closed subsector of the first or fourth quadrant, respectively,
while
\begin{equation}\label{12d}
\log B(z)\sim c z^d \exp(-z)
\end{equation}
in any closed subsector of the left half-plane.
\end{proposition}

\begin{proof}
As already mentioned, we will reduce this result to Proposition~\ref{lemma-broom1}.
Since this proposition is phrased for functions defined in the half-plane $H^+$,
we will first prove a ``half-plane version'' of our conclusion.
We will show that there exist compact sets $K$ and $K_1$,
a local homeo\-morphism $F_1\colon H^+\setminus K_1\to\C$ and $x_0>0$
such that $(H^+\setminus K_1, F_1)\sim (D\setminus K,F)$ and
\begin{equation}\label{F_1}
F_1(x)-a_\pm = e^{i\theta_\pm}\exp(-x^2)
\quad\text{for}\ \pm x\geq x_0.
\end{equation}
Here $F_1$ is quasiregular with
\begin{equation}\label{l_F1}
\logarea\!\left( \supp\!\left(\mu_{F_1}\right)\cap \Delta \right) <\infty.
\end{equation}
Moreover, there exists $c\in\C^*$ and $d\in\Z$ such that
if $\delta>0$, then, as $z\to\infty$,
\begin{equation}\label{12a}
F_1(z)-a_\pm\sim c z^d \exp(-z^2)
\end{equation}
for $\delta\leq\arg z\leq \pi/4-\delta$
and $3\pi/4+\delta\leq\arg z\leq \pi-\delta$, respectively,
while
\begin{equation}\label{12b}
\log F_1(z)\sim c z^d \exp(-z^2)
\end{equation}
for $\pi/4+\delta\leq\arg z\leq 3\pi/4-\delta$.
Defining $B$ by $B(z^2)=F_1(z)$ we then see that $B$ has the properties stated.

In order to construct the function $F_1$ we note that if $a_+\neq a_-$,
then there exist $\varepsilon>0$ and  a quasiconformal map $\psi\colon\C\to\C$ such that
\begin{equation} \label{40f}
\psi(z)=
\begin{cases}
z & \text{if}\ |z|\leq \varepsilon \ \text{or}\ |z|\geq 1/\varepsilon,\\
a_\pm + 2^{\mp 1}e^{i\theta_\pm}\left(z-2^{\pm 1}\right) & \text{if}\ \left|z-2^{\pm 1}\right|\leq \varepsilon.
\end{cases}
\end{equation}
If $a_+= a_-:=a$ and thus $\theta_+=\theta_-=:\theta$,
then we choose $\psi$ such that 
\begin{equation} \label{40f1}
\psi(z)=
\begin{cases}
z & \text{if}\ |z|\leq \varepsilon \ \text{or}\ |z|\geq 1/\varepsilon,\\
a + e^{i\theta}\left(z-1\right) & \text{if}\ \left|z-1\right|\leq \varepsilon.
\end{cases}
\end{equation}
In the first case $(D,\psi^{-1}\circ F)$ has the oriented asymptotic values
$(2^{\pm 1},0)$ at~$\pm\infty$, in the second case
$(D,\psi^{-1}\circ F)$ has the oriented asymptotic value $(1,0)$ at~$\pm\infty$.

In both cases, Proposition~\ref{lemma-broom1} is applicable to $(D,\psi^{-1}\circ F)$
and  we conclude that there exists a function $F_0$ as given in~\eqref{10a}
and compact sets $K$ and $K_0$ such that 
$(D\setminus K, \psi^{-1} \circ F)\sim (H^+\setminus K_0, F_0)$.
This means that $(D\setminus K, F)\sim (H^+\setminus K_0, \psi \circ F_0)$.

In order to prove~\eqref{F_1} and~\eqref{l_F1},
we restrict to the case $a_+\neq a_-$. The other case is analogous.
Then $F_0$ has the oriented asymptotic value $(2,0)$ at $+\infty$ and thus 
\begin{equation}\label{10f}
\exp\!\left( \int_{\xi}^\infty R_0(t)e^{- t^2}\dif t +c_0 \right)=2.
\end{equation}
For large $x$, say $x\geq x_0$, we then have
\begin{equation}\label{10g}
\begin{aligned}
\psi(F_0(x))
&= a_+ +\frac12 e^{i\theta_+} (F_0(x)-2)
= a_+ +e^{i\theta_+} \left( \exp\!\left(-\int_{x}^\infty R_0(t)e^{- t^2}\dif t \right) -1\right) .
\end{aligned}
\end{equation}
Similarly, 
\begin{equation}\label{10h}
\psi(F_0(x))
= a_- +e^{i\theta_-} \left( \exp\!\left(-\int_{-\infty}^x R_0(t)e^{- t^2}\dif t \right) -1\right) 
\end{equation}
for negative $x$ of large modulus, and we may choose $x_0$ such that this holds
for $x\leq -x_0$.
Note that since $F$ and $\psi\circ F_0$ have the oriented asymptotic values
$(a_\pm,\theta_\pm)$, this implies that $R_0(x)<0$ if $|x|$ is large.
We may assume that this holds for $|x|\geq x_0$.

We put
\begin{equation}\label{10i}
u(x)=
\begin{cases}
\exp\!\left(-\int_{x}^\infty R_0(t)e^{- t^2}\dif t \right) -1& \text{if}\ x\geq x_0,\\
\exp\!\left(-\int_{-\infty}^x R_0(t)e^{- t^2}\dif t \right) -1& \text{if}\ x\leq -x_0,
\end{cases}
\end{equation}
so that 
\begin{equation}\label{10j}
\psi(F_0(x)) = a_\pm +e^{i\theta_\pm}u(x)
\quad\text{if}\ \pm x\geq x_0.
\end{equation}
We have $u(x)>0$ and can thus define $h\colon (-\infty,-x_0]\cup[x_0,\infty)\to\R$ by
\begin{equation}\label{10j1}
\exp(-h(x)^2)=u(x),
\end{equation}
with $h(x)>0$ for $x\geq x_0$ and $h(x)<0$ if $x\leq -x_0$.
It is easy to see that for $x\geq x_0$ we have
\begin{equation}\label{10k}
u(x) \sim -\int_{x}^\infty R_0(t)e^{- t^2}\dif t \sim -\frac{R_0(x)}{2x}e^{- x^2}
\quad\text{as}\ x\to+\infty.
\end{equation}
Similarly
\begin{equation}\label{10k1}
u(x) \sim \frac{R_0(x)}{2x}e^{- x^2}
\quad\text{as}\ x\to-\infty.
\end{equation}
It follows that 
\begin{equation}\label{10l}
h(x)^2 = x^2  - \log\!\left|\frac{R_0(x)}{2x}\right| +o(1) 
\end{equation}
and hence 
\begin{equation}\label{10m}
h(x) = x +o(1) 
\quad\text{as}\ x\to\pm\infty.
\end{equation}
Moreover, a computation shows that
\begin{equation}\label{10n}
h'(x) \to 1 
\quad\text{as}\ x\to\pm\infty.
\end{equation}
Increasing the value of $x_0$ if necessary, we can extend $h$ to a 
diffeomorphism of~$\R$.
We now define $\tau\colon H^+\to  H^+$,
\begin{equation} \label{10o}
\tau(x+iy)=
\begin{cases}
x+iy+(1-y)(h(x) -x)  & \text{if}\ 0< y\leq 1 ,\\
x+iy                              & \text{if}\ y>1 .
\end{cases}
\end{equation}
For $0<y<1$ we have
\begin{equation} \label{10p}
\mu_{\tau}(z)
= \frac{(1-y)(h'(x)-1)-i(h(x)-x)}{2+(1-y)(h'(x)-1)+i(h(x)-x)} .
\end{equation}
Using~\eqref{10m} and~\eqref{10n} we see that $\tau$ is quasiconformal.

We now put $F_1=\psi\circ F_0\circ \tau^{-1}$. 
Since $\tau(x)=h(x)$ for $x\in\R$ we deduce from~\eqref{10j} and~\eqref{10j1} that
\begin{equation}\label{10r}
\begin{aligned}
F_1(x) - a_\pm 
&=\psi(F_0(\tau^{-1}(x)))-a_\pm
= e^{i\theta_\pm} u(\tau^{-1}(x))
\\ &
= e^{i\theta_\pm} \exp(-h(\tau^{-1}(x))^2)
= e^{i\theta_\pm} \exp(-x^2)
\end{aligned}
\end{equation}
for $\pm x\geq x_0$. Thus we have~\eqref{F_1}.

To prove~\eqref{l_F1} let
\begin{equation}\label{10s}
M=\left\{z\colon \varepsilon\leq |z|\leq \frac{1}{\varepsilon},\; \left|z-2^{\pm 1}\right|\geq \varepsilon\right\} .
\end{equation}
Recall here that $F_0$ has logarithmic singularities over $2^{\pm 1}$, as well as
over $0$ and~$\infty$. Thus $M$ is disjoint of the set of singular values.
To prove~\eqref{l_F1} it suffices to prove that 
\begin{equation}\label{10t}
\logarea ( F_0^{-1}(M)\cap\Delta) < \infty .
\end{equation}
In the terminology of~\cite[Definition~1.5]{Epstein2015}
we thus have to show that $F_0$ has the \emph{area property}.

Put
\begin{equation}\label{11a}
g(z)=\int_{\xi}^z R_0(t)e^{-t^2}\dif t +c_0
\end{equation}
so that $F_0=\exp g$.
Let $T_1=\{x+iy\colon 0\leq y\leq x-1\}$, $T_2=\{x+iy\colon y\geq |x|+1\}$ and
$T_3=\{x+iy\colon 0\leq y\leq -x-1\}$. We will show that 
\begin{equation}\label{11a0}
\logarea ( F_0^{-1}(M)\cap T_j) < \infty 
\end{equation}
for $1\leq j\leq 3$. This easily yields~\eqref{10t}.

Since $F_0(x)\to 2$ as $x\to\infty$ we find that $g(x)\to \log 2$ as $x\to\infty$.
Similarly as before we find that there exists $r_0$ such that
\begin{equation}\label{11b}
g(z)-\log 2 = -\int_{z}^\infty R_0(t)e^{-t^2}\dif t
\quad \text{for}\ 
0\leq \arg z \leq  \frac{\pi}{4},\; |z|\geq r_0.
\end{equation}
Here the path of integration connects $z$ with the positive axis, 
and then runs to $\infty$ along it.
To be definite, let the path of integration be the 
segments $[z,2|z|]$ and $[2|z|,\infty)$.
Integration by parts (cf.~\cite[Lemma~4.1]{Hemke2005}) shows that 
\begin{equation}\label{11c}
g(z)-\log 2 \sim \frac{R_0(z)}{2z} e^{-z^2}
\quad \text{as}\  z\to\infty,\; z\in T_1.
\end{equation}
For $z=x+iy\in T_1$ we have $\re(z^2)=x^2-y^2\geq 2x+1\geq x$ and
hence $|\exp(-z^2)|\leq e^{-x}$. It follows that 
$g(z)\to \log 2$ as $z\to\infty$, uniformly for $z\in T_1$.
Hence $F_0(z)\to 2$ as $z\to\infty$, uniformly for $z\in T_1$.
Thus $F_0^{-1}(M)\cap T_1$ is bounded so that~\eqref{11a0} holds for $j=1$.
An analogous argument shows that~\eqref{11a0} holds for $j=3$.

To show that~\eqref{11a0} also holds for $j=2$, we
note that $z\mapsto p(z):=i\sqrt{z}$ maps the right half-plane  $H$ 
conformally onto $\{z\colon \pi/4<\arg z<3\pi/4\}$.
Put $g_0=g\circ p$ and $G_0=F_0\circ p=\exp g_0$.
In view of Lemma~\ref{lemma-Aalpha} it thus suffices to show
that with $T_2'=p^{-1}(T_2)$ we have 
\begin{equation}\label{11e}
\logarea\!\left(G_0^{-1}(M)\cap T_2' \right) <\infty.
\end{equation}
With $K=\log(1/\varepsilon)$ we have 
\begin{equation}\label{11g}
G_0^{-1}(M)\subset A_K :=\{z\colon |\re g_0(z)|< K\}.
\end{equation}
Thus~\eqref{11e} will follow if we show that
\begin{equation}\label{11e1}
\logarea\!\left(A_K \cap T_2' \right) <\infty.
\end{equation}

Integration by parts yields that 
\begin{equation}\label{11d}
g(z)\sim -\frac{R_0(z)}{2z} e^{-z^2}
\quad \text{as}\  z\to\infty,\; z\in T_2.
\end{equation}
Thus there exists $\alpha\in\C^*$ and $\beta\in\R$ such that
\begin{equation}\label{11f}
g_0(z)\sim -\frac{R_0(i\sqrt{z})}{2i\sqrt{z}} e^{z}\sim \alpha z^{\beta} e^z
\quad \text{as}\  z\to\infty,\; z\in T_2'.
\end{equation}
For $z\in T_2'$ of sufficiently large modulus we may thus write 
$g_0(z)=\exp\varphi_0(z)$ with a map $\varphi_0$ satisfying
\begin{equation}\label{11h}
\varphi_0(z) =z+ \beta\log z +\log\alpha+o(1)
\quad \text{as}\  z\to\infty,\; z\in T_2'.
\end{equation}
With $T_2''=\{z\in T_2'\colon \dist(z,\partial T_2')\geq 1\}$ this yields that
\begin{equation}\label{11i}
\varphi_0'(z) \to 1
\quad \text{as}\  z\to\infty,\; z\in T_2''.
\end{equation}
Hence for $C\subset T_2''$ we have $\logarea \varphi_0(C)<\infty$ if
and only if $\logarea C<\infty$. 
To prove~\eqref{11e1} it thus suffices to show that 
\begin{equation}\label{11i1}
C=\{z\colon \re z>1,\; |\re(e^z)|\leq K\}
\end{equation}
has finite logarithmic area. In order to do so we put, for $k\in\Z$,
\begin{equation}\label{11i2}
S_k=\{x+iy\colon x\geq 1, \; k\pi\leq y\leq (k+1)\pi\}. 
\end{equation}
For $z=x+iy\in S_k$ we have $z\in C$ if $|\cos y|=|\sin(y-(k+1/2)\pi)|
\leq e^{-x}$. Hence 
\begin{equation}\label{11j}
S_k\cap C\subset S_k':=
\left\{ x+iy\colon\left|y-\left(k+\frac12\right)\pi\right|\leq e^{-x}\right\}.
\end{equation}
For $k\geq 1$ and $z=x+iy\in S_k'$ we have $x^2+y^2\geq y^2\geq k^2$. Thus  
\[
\begin{aligned}
\logarea(S_k')
&=\int_{S_k'}\frac{\dif x\,\dif y}{x^2+y^2}
\leq \frac{1}{k^2} \int_1^{\infty} \int_{\pi/2+\pi k -e^{-x}}^{\pi/2+\pi k +e^{-x}} \dif y\,\dif x
= \frac{2}{k^2} \int_1^{\infty} e^{-x} \dif x
= \frac{2}{e k^2}  .
\end{aligned}
\]
The same argument yields that if $k\leq -2$, then $\logarea(S_k')\leq 2/(e(k+1)^2)$. 
Overall we obtain
\begin{equation}\label{11k}
\begin{aligned}
\logarea C
&\leq \sum_{k=-\infty}^\infty \logarea(S_k')
\leq 
\frac{4}{e}\sum_{k=1}^\infty \frac{1}{k^2}
+ \logarea\{x+iy\colon x\geq 1,\;|y|\leq \pi\} 
<\infty.
\end{aligned}
\end{equation}
This completes the proof of~\eqref{11e1} and hence~\eqref{11e},
finishing the proof of~\eqref{F_1}.

Finally, it follows from the definition of $F_1$ and~\eqref{11c} that
\begin{equation}\label{11l}
\begin{aligned}
F_1(z)-a_+
&= \psi(F_0(z))-a_+
=\frac12 e^{i\theta_+} (F_0(z)-2) 
= e^{i\theta_+} (\exp(g(z)-\log 2) -1)
\\ &
\sim   e^{i\theta_+}( g(z)-\log 2 )
\sim  -e^{i\theta_+} \frac{R_0(z)}{2z} e^{-z^2}
\quad \text{as}\  z\to\infty,\; z\in T_1.
\end{aligned}
\end{equation}
This yields the asymptotics for $F_1(z)-a_+$ given in~\eqref{12a}.
Those for $F_1(z)-a_-$ are obtained analogously.
Finally, \eqref{12b} follows from~\eqref{11d} in the same fashion.
This completes the proof of Proposition~\ref{lemma-broom3}.
\end{proof}

\subsection{Elements with infinitely many zeros and poles}\label{elem-zeros}
We want to glue the element $(\Omega\setminus K',B)$ from Proposition~\ref{lemma-broom3}
to (a modification of) the tangent map.
Instead of the tangent we will, for $a\in\C\setminus\R$, consider the restriction of
\begin{equation} \label{40k}
v_a(z)=
\tan\!\left(\frac{z}{2}\right) \im a +\re a
\end{equation}
to the right half-plane~$H$.
The element $(H,v_a)$ 
has the oriented asymptotic values $(a,-\sign(\im a)\pi/2)$ at
$-\infty$ and $(\overline{a},\sign(\im a)\pi/2)$ at $+\infty$.
This yields that $(H,v_a)$ can be glued to $(\Omega\setminus K',B)$ if
\begin{equation} \label{40k1}
d_+=(a,-\sign(\im a)\pi/2)
\quad\text{or}\quad
d_-=(\overline{a},\sign(\im a)\pi/2).
\end{equation}
To make this gluing explicit we modify the map~$v_a$.
Note that
\begin{equation} \label{40l}
v_a(iy)= a+\left(\tan\!\left(\frac{iy}{2}\right) -  i\right) \im a = a - i\frac{ 2\im a}{e^{ y}+1}
\end{equation}
as well as
\begin{equation} \label{40m}
v_a(iy) =\overline{a}+\left(\tan\!\left(\frac{iy}{2}\right) +  i\right) \im a
=\overline{a} + i\frac{ 2\im a}{e^{-y}+1} .
\end{equation}
Our aim is to construct a quasiconformal map $\phi$ such that
if~\eqref{40k1} holds, then
there exists 
$t_1\geq t_0$ such that
\begin{equation} \label{40m1}
\begin{aligned}
(v_a\circ\phi)(\gamma_H(-t))) -a
&=
v_a(\phi(it))-a
=B(\gamma_\Omega(t))-a_+
\quad\text{for}\ t\geq t_1
\end{aligned}
\end{equation}
or
\begin{equation} \label{40m2}
\begin{aligned}
(v_a\circ\phi)(\gamma_H(-t))) -\overline{a}
&=
v_a(\phi(it))-\overline{a}
=B(\gamma_\Omega(t))-a_-
\quad\text{for}\ t\leq -t_1 ,
\end{aligned}
\end{equation}
respectively.

To construct the map~$\phi$, put $y_a=\max\{1,-\log|\im a|\}$ and define
\begin{equation} \label{40n}
q_a\colon [y_a,\infty)\to\R,
\quad
q_a(y) = \log\!\left(2|\im a|e^y -1 \right).
\end{equation}
Then
\begin{equation} \label{40o}
\begin{aligned}
q_a(y)
=y+\log (2|\im a|) +o(1)
\quad\text{as}\ y\to +\infty
\end{aligned}
\end{equation}
and
\begin{equation} \label{40p}
q_a'(y) =1+o(1) \quad\text{as}\ y\to +\infty.
\end{equation}
Thus there exists a diffeomorphism $Q_a\colon \R\to\R$ such that
\begin{equation} \label{40q}
Q_a(y)=
\begin{cases}
q_a(y) & \text{if}\ y\geq y_a ,\\
-q_a(-y) & \text{if}\ y\leq -y_a .
\end{cases}
\end{equation}
We now define $\phi=\phi_a\colon H\to  H$,
\begin{equation} \label{40r}
\phi_a(x+iy)=
\begin{cases}
x+iy+i(1-x)(Q_a(y) -y)  & \text{if}\ 0< x\leq 1 ,\\
x+iy                              & \text{if}\ x>1 .
\end{cases}
\end{equation}
Thus $\phi_a(iy)=iQ_a(y)$.
A computation analogous to~\eqref{10p} shows together with~\eqref{40o}
and~\eqref{40p} that $\phi_a$ is quasiconformal.

We now define $T_{a}\colon H\to \C$,
\begin{equation} \label{40t}
T_a(z)=v_a(\phi_{a}(z)).
\end{equation}
\begin{definition}\label{def-tan}
An element $(D,F)$ is of type $\Tan_a$ if there exist compact
sets $K$ and $K'$ such that $(D\setminus K,F)\sim (H\setminus K',T_a)$.
\end{definition}
Of course, if $(D,F)$ is of type $\Tan_a$, then also 
 $(D\setminus K,F)\sim (H\setminus K'',v_a)$ for some compact set~$K''$.

\begin{lemma}\label{lemma-Ta}
Let $d_\pm=(a_\pm,\theta_\pm)$, $(D,F)$, $B$ and $t_0$ be as
in Proposition~\ref{lemma-broom3} and let $T_a$ be defined by~\eqref{40t}.
Let $t_1=\max\{t_0,y_a\}$.
\begin{itemize}
\item[$(i)$]
If $d_+=(a,-\sign(\im a)\pi/2)$, then
\begin{equation} \label{40m3}
T_a(\gamma_H(-t))) -a =B(\gamma_\Omega(t))-a_+
\quad\text{for}\ t\geq t_1 .
\end{equation}
\item[$(ii)$]
If $d_-=(\overline{a},\sign(\im a)\pi/2)$, then
\begin{equation} \label{40m4}
T_a(\gamma_H(-t))) -\overline{a}
=B(\gamma_\Omega(t))-a_-
\quad\text{for}\ t\leq -t_1 .
\end{equation}
\end{itemize}
\end{lemma}
\begin{proof}
To prove $(i)$ we note that if $t\geq t_1$, then
by~\eqref{40t}, \eqref{40r}, \eqref{40q}, \eqref{40l}, \eqref{40n} and  \eqref{10x}
we have
\begin{equation} \label{40u1}
\begin{aligned}
T_a(\gamma_H(-t)) -a
&
= T_a(it) -a
= (v_a\circ \phi_a)(it) -a
= v_a(i Q_a(t)) -a
\\ &
= v_a(i q_a(t)) -a
=- i\frac{2\im a}{\exp{q_a(t)}+1}
=   -i\sign(\im a)\exp(-t)
\\ &
=  e^{-i\sign(\im a)\pi/2}\exp(-t)
= B(\gamma_\Omega(t)) - a_+ .
\end{aligned}
\end{equation}
The proof of $(ii)$ is analogous.
\end{proof}
Finally we note that because $T_a$ is conformal except in the 
vertical strip $\{z\colon 0<\re z<1\}$, we have
\begin{equation}\label{l_Ta}
\logarea\!\left( \supp\!\left(\mu_{T_a}\right)\cap \Delta \right) <\infty.
\end{equation}

\section{Beginning of the proof of Theorem~\ref{theorem1}: Cutting into pieces} \label{section8}
Let $F$ be a symmetric local homeo\-morphism of class $S$ with a finite number $m$ of
singularities of $F^{-1}$ over $\C^*$.

Suppose first that $m=0$. Then $F\colon\C\to\C^*$ is 
a covering. This leads easily to the following result.
\begin{lemma}\label{lemma-m0}
Let $F\colon \C\to\bC$ be a local homeo\-morphism
such that the inverse $F^{-1}$ has no  singularities
over points in~$\C^*$. Then $(\C,F)\sim (\C,\exp)$. 
\end{lemma}
Since the inverse of the exponential function has only one
singularity over $0$ and~$\infty$, 
while our hypothesis says that $F^{-1}$ has infinitely many singularities
over $0$ and~$\infty$, we deduce that the case $m=0$ does not occur.
Thus $m\geq 1$.

Next we show that $F^{-1}$ actually has infinitely many singularities over
both~$0$ and~$\infty$. To do so, we may assume 
without loss of generality that $F^{-1}$ has infinitely many singularities over~$0$. 
Lemma~\ref{lemma-lindeloef} yields that $F$ is unbounded 
in the region ``between'' two tracts over~$0$. Since all poles are real, this means that 
with at most two exceptions the region between two tracts over $0$ must contain
a tract over~$\infty$.

To each asymptotic value
$a\in \C^*$ we associate a semi-open segment
$\ell_a=(a,a+\varepsilon e^{i\theta(a)}]$.
Here $\varepsilon>0$ is chosen so small that the
disks $D(a,2\varepsilon)$ are disjoint and do not contain~$0$.
The choice of $\theta(a)\in\{0,\pm\pi/2,\pi\}$ will be fixed later.

For an asymptotic value~$a\in \C^*$, there exists at least one
unbounded component $U$ of $F^{-1}(D(a,2\varepsilon))$.
For each such component~$U$, the map
$F\colon U\to D(a,2\varepsilon)\setminus\{ a\}$ is a universal cover.
Hence $U$ contains  a component $\gamma_U$ of $F^{-1}(\ell_a)$.
Clearly $\gamma_U$ is a curve tending to $\infty$ in~$U$.
For each component $U$ we fix such a curve~$\gamma_U$.
So overall there are $m$ such curves and they are disjoint.
If $U$ is symmetric, then we may choose $\gamma_U$ as a subset of~$\R$.
Then $\theta(a)\in\{0,\pi\}$.
If symmetry interchanges two components, then we choose the corresponding
angles and curves such that $\theta(\overline{a})=-\theta(a)$ 
and $\gamma_{\overline{U}} =\overline{\gamma_U}$.

If $F$ has a two-sided sequence of zeros and poles,
then no such curve $\gamma_U$ can be contained in~$\R$.
Hence the set of these curves splits into complex conjugate pairs.
It follows that the number $m$ must be even in this case.

Since the curves $\gamma_U$ tend to $\infty$, there is
a cyclic order on the set of these curves.
We enumerate them counterclockwise.
This enumeration is independent of the choice of the angles~$\theta(a)$.
If there are infinitely many positive zeros and poles, we do it in
such a way that the positive ray is between $\gamma_m$ and $\gamma_1$.
Let $a_j$ be the asymptotic value along $\gamma_j$. By symmetry,
we have $a_j=\overline{a_{m-j+1}}$
and $\gamma_j=\overline{\gamma_{m-j+1}}$ for $1\leq j\leq m$.
We put $\gamma_0=\gamma_m$ and $a_0=a_m$.

\begin{figure}[!ht]
\captionsetup{width=.9\textwidth}
\def\svgwidth{\textwidth}
\begingroup%
  \makeatletter%
  \providecommand\color[2][]{%
    \errmessage{(Inkscape) Color is used for the text in Inkscape, but the package 'color.sty' is not loaded}%
    \renewcommand\color[2][]{}%
  }%
  \providecommand\transparent[1]{%
    \errmessage{(Inkscape) Transparency is used (non-zero) for the text in Inkscape, but the package 'transparent.sty' is not loaded}%
    \renewcommand\transparent[1]{}%
  }%
  \providecommand\rotatebox[2]{#2}%
  \newcommand*\fsize{\dimexpr\f@size pt\relax}%
  \newcommand*\lineheight[1]{\fontsize{\fsize}{#1\fsize}\selectfont}%
  \ifx\svgwidth\undefined%
    \setlength{\unitlength}{1200bp}%
    \ifx\svgscale\undefined%
      \relax%
    \else%
      \setlength{\unitlength}{\unitlength * \real{\svgscale}}%
    \fi%
  \else%
    \setlength{\unitlength}{\svgwidth}%
  \fi%
  \global\let\svgwidth\undefined%
  \global\let\svgscale\undefined%
  \makeatother%
  \begin{picture}(1,0.16875)%
    \lineheight{1}%
    \setlength\tabcolsep{0pt}%
    \put(0,0){\includegraphics[width=\unitlength,page=1]{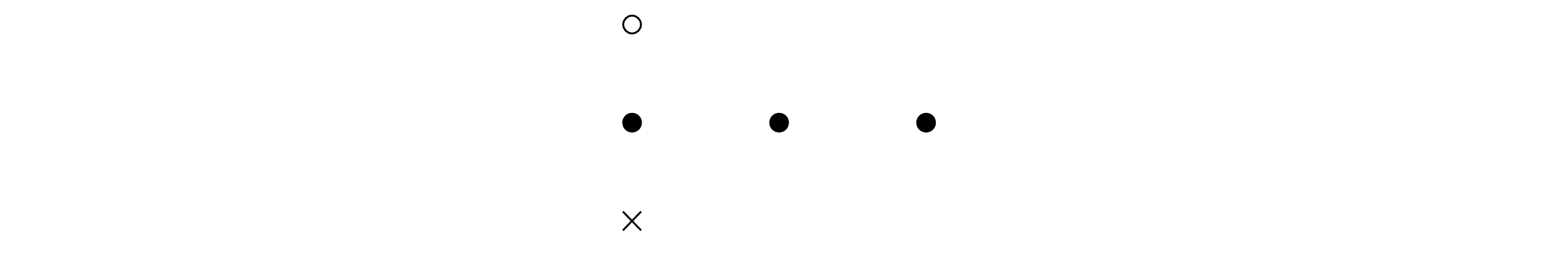}}%
    \put(0.39611815,0.0610547){\color[rgb]{0,0,0}\makebox(0,0)[lt]{\lineheight{1.25}\smash{\begin{tabular}[t]{l}$0$\end{tabular}}}}%
    \put(0.40236818,0.12355469){\color[rgb]{0,0,0}\makebox(0,0)[lt]{\lineheight{1.25}\smash{\begin{tabular}[t]{l}$i$\end{tabular}}}}%
    \put(0.39299318,0.0016797){\color[rgb]{0,0,0}\makebox(0,0)[lt]{\lineheight{1.25}\smash{\begin{tabular}[t]{l}$-i$\end{tabular}}}}%
    \put(0.48986815,0.06417969){\color[rgb]{0,0,0}\makebox(0,0)[lt]{\lineheight{1.25}\smash{\begin{tabular}[t]{l}$a$\end{tabular}}}}%
    \put(0.58049314,0.06417969){\color[rgb]{0,0,0}\makebox(0,0)[lt]{\lineheight{1.25}\smash{\begin{tabular}[t]{l}$\infty$\end{tabular}}}}%
    \put(0,0){\includegraphics[width=\unitlength,page=2]{fig3.pdf}}%
  \end{picture}%
\endgroup%

\caption{The symmetric cell decomposition $C_0$ in the proof of Lemma~\ref{lemma1-a}.\label{fig:symmetric_cell_decomposition}}
\end{figure}

\begin{lemma}\label{lemma1-a}
If $F$ has infinitely many positive zeros and poles, then $a_1\neq a_m$
so that $a_1=\overline{a_m}\notin\R$ and $m\geq 2$.

If, in addition, $F$ has infinitely many negative zeros and poles, then 
$m$ is even, $a_{m/2}\neq a_{m/2+1}$ and $a_{m/2} =\overline{a_{m/2+1}}\notin\R$.
\end{lemma}
\begin{proof}
Suppose that $a_1=a_m$ so that $a:=a_1=\overline{a_m}=\overline{a_1}\in\R$.
Consider a symmetric cell decomposition $C_0$ of the sphere $\bC$
with two vertices $\times =i$ and $\circ =-i$, three edges and three faces
containing $0$, $\infty$ and~$a$, respectively; 
see Figure~\ref{fig:symmetric_cell_decomposition}. 
We may assume that the edges do not intersect any of the disks $D(a_j,2\varepsilon)$.
Thus all these disks are contained in some face of~$C_0$.
Let $L_0=F^{-1}(C_0)$.

Since $F$ is real and locally univalent, $F$ is either always decreasing
or always increasing between two adjacent poles. This implies that if $x_0$ is 
the smallest positive pole and if $(x_k)$ denotes the sequence of all poles,
zeros and $a$-points greater than or equal to~$x_0$, ordered such that
$x_0<x_1<x_2<\dots$, then $x_k$ is a pole if 
$k$ is divisible by~$3$, and either $x_k$ is a zero if $k\equiv 1\pmod{3}$ and
an $a$-point if $k\equiv 2\pmod{3}$, or vice versa.

Let $X_k$ be the face of $L_0$ containing $x_k$ and let $Y_k=F(X_k)$ be the 
corresponding face of $C_0$.
Suppose that  $F$ is not a covering from $X_k\setminus\{x_k\}$ to $Y_k\setminus \{F(x_k)\}$.
Then there exists $j\in\{1,\dots,m\}$ such that
$\gamma_j \subset X_k$ and $D(a_j,2\varepsilon)\subset Y_k$.
This can happen for at most $m$ values of $k$. Thus there exists $K\in\N$ such that 
$F$ is a  covering from $X_k\setminus\{x_k\}$ to $Y_k\setminus\{F(x_k)\}$,
with $F(x_k)\in\{0,a,\infty\}$.
It follows that $F$ maps $X_k$ univalently to~$Y_k$.
This implies $X_k$ is a digon for all $k\geq K$.
Now $X_K$ shares an edge with $X_{K+1}$, and these two edges end at the same vertices.
Next, $X_{K+1}$ shares an edge with $X_{K+2}$, and again these two edges end at the same vertices.
We conclude that the edges of $X_K$, $X_{K+1}$ and $X_{K+2}$ end all at the same vertices.
This is a contradiction, since only three edges can meet at a vertex.

The conclusion that $a_{m/2}\neq a_{m/2}+1$ if $F$ has infinitely many negative
zeros and poles follows by considering $F(-z)$ instead of $F(z)$.
\end{proof}

Let $\theta_j=\theta(a_j)$.
Lemma~\ref{lemma1-a} says that if $F$ has infinitely many positive zeros
and poles, then $\im a_1=-\im a_m\neq 0$.
In order to apply Lemma~\ref{lemma-Ta}, we choose $\theta_1=-\sign(\im a_1)\pi/2$
and $\theta_m=-\theta_1$.
Similarly, if $F$ has infinitely many negative zeros
and poles, then we choose $\theta_{m/2}=-\sign(\im a_{m/2})\pi/2$ 
and $\theta_{m/2+1}=-\theta_{m/2}$.
If $\gamma_j$ is contained in the real axis, 
then, as already mentioned, we have $\theta_j=0$ or $\theta_j=\pi$.
For all other $j$ the choice of $\theta_j$ is irrelevant, but to be definite
we choose $\theta_j=0$ for these~$j$.
We put $d_j=(a_j,\theta_j)$.

We connect the endpoints of $\gamma_j$ to some point in $\R$ by 
curves $\sigma_j$ which are pairwise disjoint except for their common
endpoint in $\R$. Moreover, we assume that $\sigma_j$ intersects $\gamma_j$
only at its
endpoint, and does not intersect any other~$\gamma_k$.
Then, for $1\leq j\leq m$, there exists an unbounded domain $G_j$ 
whose boundary is formed by the curves 
$\gamma_{j-1}$, $\gamma_j$, $\sigma_{j-1}$ and~$\sigma_j$.

The proof of Theorem~\ref{theorem1} is split into two parts: 
The first part (given in this section) is a purely
topological statement, and the second part (in the next section)
deals with the asymptotic behavior.

The topological statement is the following.

\begin{theorem}\label{theorem2}
Let $F\colon \C\to\bC$ be a symmetric local homeo\-morphism
such that the inverse $F^{-1}$ has a finite, non-zero number $m$ of singularities
over points in~$\C^*$.
Then:
\begin{itemize}
\item[$(i')$]
If $F$ has only finitely many zeros and poles, then 
$(G_j,F)$ is of type $\Broom_{d_j,d_{j+1}}$ for all~$j$.

\item[$(ii')$]
If $F$ has infinitely many zeros and poles, all of them
positive, then $m\geq 2$ and $(G_0,F)$ is of type $\Tan_{a_1}$
while $(G_j,F)$ is of type $\Broom_{d_j,d_{j+1}}$ for $1\leq j\leq m-1$.

\item[$(iii')$]
If $F$ has a two-sided sequence of zeros and poles, all of them real,
then $m$ is even and $m\geq 2$. Moreover,
$(G_0,F)$ is of type $\Tan_{a_1}$ and $(G_{m/2},F)$ is of type $\Tan_{a_{m/2+1}}$
while $(G_j,F)$ is of type $\Broom_{d_j,d_{j+1}}$ for all other~$j$.
\end{itemize}
\end{theorem}

\begin{proof}
Definition~\ref{def-broom} says that if $F$ has no zeros and 
poles in~$G_j$, then $(G_j,F)$ is of type $\Broom_{d_j,d_{j+1}}$.
This already proves $(i')$ and also shows that in case
$(ii')$ the $(G_j,F)$ are of the stated type if $j\neq 0$ and that in case
$(iii')$ they are of the stated type if $j\neq 0$ and $j\neq m/2$.

To deal with $(G_0,F)$ in cases $(ii')$ and $(iii')$, 
we consider a symmetric cell decomposition $C_1$ of the sphere
with two vertices $\times$ and $\circ$ on the real line, four edges and four faces
containing $0$, $\infty$, $a:=a_1$ and $\overline{a}=a_m$, respectively, see
the left part of Figure~\ref{figure2}.
According to Lemma~\ref{lemma1-a}, we have $a_1\neq a_m$, which justifies this
construction.

\begin{figure}[!ht]
\captionsetup{width=.9\textwidth}
\def\svgwidth{\textwidth}
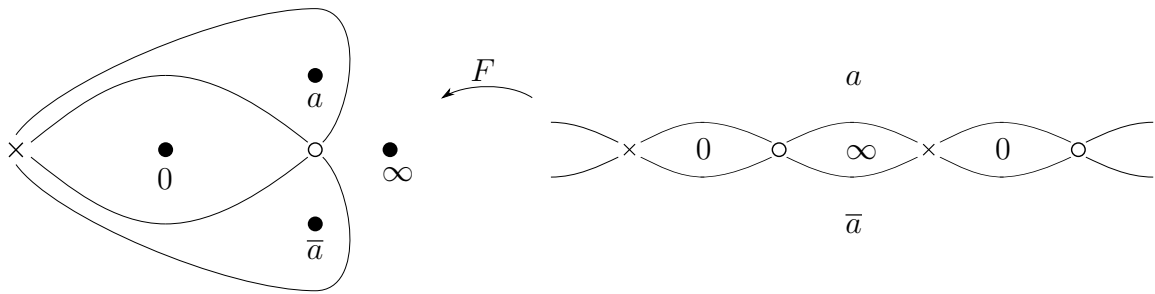
 \caption{A cell decomposition for singular values $0$, $a$, $\overline{a}$ and $\infty$.}
\label{figure2} 
\end{figure}

We may assume that if $a_j$ is an asymptotic value with $a_j\notin\{a,\overline{a}\}$,
then the disk $D(a_j,2\varepsilon)$ is contained in one of the
faces labeled $0$ and~$\infty$. Let $L_1=F^{-1}(C_1)$. Then the restriction of $F$ to
an unbounded face of $L_1$ labeled by
$a$ or $\overline{a}$ is a universal cover from this unbounded face 
to the corresponding face of $C_1$ punctured at $a$ or $\overline{a}$, respectively.

The same argument as in the proof of Lemma~\ref{lemma1-a} shows that 
all sufficiently large positive poles and zeros of $F$
must belong to digons of $L_1$.
Each such digon has both vertices on the positive ray.
Conversely, every sufficiently large positive vertex belongs to two digons, one labeled
$0$ and one labeled~$\infty$.

Thus there is a one-sided chain of the form $\times =\circ =\times =\circ =\times =\dots$ infinite
in the positive direction, consisting of faces labeled $0$ and $\infty$
alternatively.
Consideration of the cyclic order of face labeling around vertices of these digons containing
zeros and poles shows that
 immediately above and below this chain
we must have $\infty$-gons labeled $a$ and $\overline{a}$; cf.\ the right part of Figure~\ref{figure2}.

These $\infty$-gons must contain $\gamma_1$ and $\gamma_m$, since there are no other
asymptotic curves $\gamma_j$ with asymptotic values in $\C^*$
between $\gamma_m$ and $\gamma_1$ in the sense of the counterclockwise
cyclic order at $\infty$. Let $D$ be the region
consisting of these two $\infty$-gons and the closure of digons of the chain.

Recall that $T_{a_1}$ is defined in~\eqref{40t} by $T_{a_1}=v_{a_1}\circ \phi_{a_1}$,
with $v_{a_1}$ given by~\eqref{40k} and a quasiconformal map~$\phi_{a_1}$.
The cell decomposition $v_{a_1}^{-1}(C_1)$ consists of
a two-sided infinite chain of digons the form $\ldots = \times =\circ=\times =\circ= \ldots$
and two $\infty$-gons. Removing the closure of the left part
of this infinite chain, we obtain a region $D'$ with a cell decomposition
combinatorially equivalent to the restriction of the
cell decomposition $L_1$ on~$D$. Therefore, the restriction of $F$ to $D$
is equivalent to the restriction of $v_{a_1}$ to~$D'$,
that is $v_{a_1}=F\circ\psi$, where $\psi\colon  D'\to D$ is a homeo\-morphism.
Since $F$ and $v_{a_1}$ are symmetric, $\psi$ can also be chosen to be symmetric.
Moreover, since $\gamma_m=\overline{\gamma_1}$ and since $F$ maps the curves
$\gamma_1$ and $\gamma_m$ onto the segments 
$(a_1,a_1+\varepsilon e^{i\theta_1}]=(a,a-\sign(\im a)i]$ and 
$(a_m,a_m+\varepsilon e^{i\theta_m}]=(\overline{a},\overline{a}+\sign(\im a)i]$,
we see that there exists $k\in\Z$ such that 
$\psi^{-1}(\gamma_1)$ and $\psi^{-1}(\gamma_m)$ are contained in
the line $\{z\colon \re z=k\pi\}$.
There is no loss of generality to assume that they are on the imaginary axis.

Now we extend the curves $\gamma_1$ and $\gamma_m$ until they hit
some vertices on the boundary of their $\infty$-gons. 
This can be done in such a way that $\psi^{-1}(\gamma_1)$ and $\psi^{-1}(\gamma_m)$ 
are contained in the right half-plane.
These extended curves cut from $D$ a region $D_0$ which contains a positive ray.
The restriction of $F$ on $D_0$ is equivalent to the restriction
of $v_{a_1}$ onto the region $\psi^{-1}(D_0)$, 
and hence the restriction of $T_{a_1}$ onto $\phi^{-1}(\psi^{-1}(D_0))$.
This yields that $(G_0,F)$ is of type~$\Tan_{a_1}$.
This completes the proof of $(ii')$ and also handles the case of $(G_0,F)$ in
case $(iii')$.

To complete the proof in case $(iii')$  we only have to note that 
considering $F(-z)$ instead of $F(z)$ corresponds to interchanging $G_0$ and $G_{m/2}$
as well as $a_1$ and $a_{m/2+1}$.
\end{proof}

\section{Completion of the proof of Theorem~\ref{theorem1}: Gluing pieces together} \label{section9}
Let $(i')-(iii')$ be the cases considered in Theorem~\ref{theorem2}.
Of course, these correspond to the cases 
$(i)-(iii)$ of Theorem~\ref{theorem1}.
We divide the plane into $m$ sectors $S_0,\dots,S_{m-1}$, enumerated counterclockwise
and such that $S_0$ is bisected by the positive real axis.
Let $\sigma_j$ be the opening angle of $S_j$.
We choose $\rho$ and the opening angles $\sigma_j$ as follows:
\begin{enumerate}
\item[]Case $(i')$:
$\rho=m$ and $\sigma_j=2\pi/\rho$ for all $j$.
\item[]Case $(ii')$:
 $\rho=m-1/2$, $\sigma_0=\pi/\rho$ and $\sigma_j=2\pi/\rho$ for $1\leq j\leq m-1$.
\item[]Case $(iii')$:
$\rho=m-1$, $\sigma_0=\sigma_{m/2}=\pi/\rho$ and $\sigma_j=2\pi/\rho$ for 
all other $j$.
\end{enumerate}
Recall here that $m$ is even in case $(iii')$ by Theorem~\ref{theorem2}.
We call a sector $S_j$ \emph{large} if $\sigma_j=2\pi/\rho$ and 
\emph{small} if $\sigma_j=\pi/\rho$.

If $S_j$ is a large sector, then there exists $e_j\in\C$ with $|e_j|=1$ such
that $z\mapsto e_j z^\rho$ maps $S_j$ conformally onto $\Omega^0=\C\setminus [0,\infty)$. In fact, we
have $e_j=-1$ for all $j$ in case $(i')$ and $e_j=-i$ for all $j$ in case $(ii')$.
In case $(iii')$ we have $e_j=-i$ for $1\leq j\leq m/2-1$ and $e_j=i$ for $m/2+1\leq j\leq m-1$.

If $S_0$ is a small sector, then $z\mapsto z^\rho$ maps $S_0$ onto the right half-plane~$H$.
If $S_{m/2}$ is a small sector, which happens only in case $(iii')$, then
$z\mapsto - z^\rho$ maps $S_0$ onto the right half-plane.
Putting $e_0=1$ and $e_{m/2}=-1$ we see that if $S_j$ is a small sector
and thus $j\in\{0,m/2\}$,
then $z\mapsto e_j z^\rho$ maps $S_j$ to~$H$.
Let $p_j\colon S_j\to \C$, $p_j(z)=e_jz^\rho$. Then $p_j(S_j)=\Omega^0$
or $p_j(S_j)=H$, depending on whether $S_j$ is large or small.

Let  $d_j=(a_j,\theta_j)$ be as in section~\ref{section8} and
let $G_0,\dots,G_{m-1}$ be the domains defined before Theorem~\ref{theorem2}.
By Theorem~\ref{theorem2}, each element $(G_j,F)$ is of one of two types.
If it is of type $\Broom_{d_j,d_{j+1}}$,  we choose the map $B_j$ 
and compact sets $K_j$ and $K_j'$ according to Proposition~\ref{lemma-broom3}
so that $(\Omega\setminus K_j',B_j)\sim (G_j\setminus K_j,F)$.
Otherwise there are compact sets $K_j$ and $K_j'$ such that 
so that $(H\setminus K_j',T_{a_{j+1}})\sim (G_j\setminus K_j,F)$.
Note that our labeling of the sectors is such that 
\begin{equation}\label{7b1}
(G_j\setminus K_j, F)\sim 
\begin{cases}
(\Omega\setminus K_j',B_j)  & \text{if}\ S_j \ \text{is large}, \\
(H\setminus K_j',T_{a_{j+1}})      & \text{if}\ S_j \ \text{is small}.
\end{cases}
\end{equation}

We consider the map 
$F_1\colon\bigcup_{j=0}^{m-1} S_j\to \C$,
which for $z\in S_j$ of sufficiently large modulus is  defined by
\begin{equation}\label{7b}
F_1(z)=
\begin{cases}
B_j(p_j(z)) & \text{if}\ S_j \ \text{is large}, \\
T_{a_{j+1}}(p_j(z)) & \text{if}\ S_j \ \text{is small}. 
\end{cases}
\end{equation}
Proposition~\ref{lemma-broom3} and Lemma~\ref{lemma-Ta} yield that, apart
from some bounded set, the expressions defining $F_1$ match on the boundaries of
the sectors. Thus there
exists $R>0$ such that~\eqref{7b} defines a quasiregular map $F_1\colon\{z\colon |z|>R\}\to\C$.

By~\eqref{7b1} there exist homeo\-morphisms $\phi_j\colon p_j(S_j)\setminus K_j'\to G_j\setminus K_j$
such that 
\begin{equation}\label{7c}
F(\phi_j(z))=
\begin{cases}
B_j(z) & \text{if}\ S_j \ \text{is large}, \\
T_{a_{j+1}}(z)     & \text{if}\ S_j \ \text{is small}. 
\end{cases}
\end{equation}
Let $\tau_j=\phi_j\circ p_j$.
Then
\begin{equation}\label{7d}
\begin{aligned}
F(\tau_j(z))
&=
F(\phi_j(p_j(z)))
=
\begin{cases}
B_j(p_j(z))) & \text{if}\ S_j \ \text{is large} \\
T_{a_{j+1}}(p_j(z)))     & \text{if}\ S_j \ \text{is small} 
\end{cases}
\\ &
=F_1(z)
\quad\text{for}\ z\in S_j\setminus p_j^{-1}(K_j)
\end{aligned}
\end{equation}
by~\eqref{7b} and~\eqref{7c}.
Hence the $\tau_j$ can be glued together 
to yield compact sets $K$ and $K'$ and a homeo\-morphism 
$\tau\colon \C\setminus K'\to\C\setminus K$ such that 
\begin{equation}\label{7e}
F(\tau(z))=F_1(z)
\quad\text{for}\ z\in \C\setminus K'.
\end{equation}

On the other hand, as explained in \S~\ref{topo-hol}, the Uniformization 
Theorem yields that there exists $0<R\leq\infty$, a homeo\-morphism 
$\phi_0\colon\C\to D(0,R)$ and a meromorphic function 
$F_0\colon D(0,R)\to\C$ such that $F=F_0\circ \phi_0$. 
With $\alpha=\phi_0\circ \tau$ we thus have 
\begin{equation}\label{7e0}
F_0(\alpha(z))=F_1(z)
\quad\text{for}\ z\in \C\setminus K'.
\end{equation}
Since $F_0$ is meromorphic and $F_1$ is quasiregular, we find that $\alpha$ is quasiconformal.
Since a quasiconformal map distorts the modulus of an annulus only by a bounded factor,
this implies that $R=\infty$.

The set where $\alpha$ is not conformal agrees with the set where $F_1$ is not meromorphic.
By Lemma~\ref{lemma-Aalpha},
\eqref{l_B} and~\eqref{l_Ta} this set has finite logarithmic area.
It thus follows from the Teich\-m\"uller--Wit\-tich--Belinskii theorem (Lemma~\ref{lemma-twb})
that there exists $a\in\C^*$ such that 
\begin{equation}\label{7f}
\alpha(z)\sim a z 
\quad\text{as}\ z\to\infty.
\end{equation}
It will be convenient to consider the inverse $\beta=\alpha^{-1}$. With $b=1/a$ 
we then have 
\begin{equation}\label{7e1}
F_0(z)=F_1(\beta(z))
\quad\text{for}\ z\in \C\setminus K,
\end{equation}
with
\begin{equation}\label{7f1}
\beta(z)\sim b z 
\quad\text{as}\ z\to\infty.
\end{equation}
As $F_0$ and $F_1$ are symmetric, $\beta$ is also symmetric. This implies 
that $b$ is real. In fact, we may assume that $b>0$.

It follows from the definition of the $T_{a_j}$ and $B_j$, 
\eqref{7b}, \eqref{7e1} and~\eqref{7f1} that 
\begin{equation}\label{7f2}
\log m(r,F_0)=\LO(r^\rho)
\end{equation}
as $r\to\infty$.
Moreover, we find that if $F_0$ has infinitely many zeros and poles, then
there exists a positive constant $C$ such that
\begin{equation}\label{7f3}
N(r,F_0)\sim C r^\rho 
\quad\text{and}\quad
N\!\left(r,\frac{1}{F_0}\right)\sim C r^\rho 
\end{equation}
as $r\to\infty$.
In fact, we have $C=b^\rho/(2\pi)$ if the sequence of zeros and poles is one-sided and
$C=b^\rho/\pi$ if it is two-sided.
It follows from these equations and the lemma on the logarithmic
derivative~\cite[Chapter~3, \S~1]{GO} that 
$F_0'/F_0$ and hence $E=F_0/F_0'$ have order~$\rho$. 
Moreover, it follows that $\lambda(E)=\rho$ in cases $(ii)$ and $(iii)$.

To prove that $E$ is of completely regular growth, we first consider a large sector $S_j$.
Let $S_j'$  be a subsector of $S_j$ which is mapped
to a subsector of the left half-plane under~$p_j$.
By~\eqref{12d} we have
\begin{equation}\label{13a}
\begin{aligned}
\log F_0(z)
&= \log B_j(p_j(\beta(z)))) 
\sim  c p_j(\beta(z))^d \exp(-p_j(\beta(z)))
\quad \text{as}\ z\to\infty, \; z\in S_j'
\end{aligned}
\end{equation}
and thus 
\begin{equation}\label{13b}
\log \log F_0(z)
\sim -p_j(\beta(z))
\sim -e_j(bz)^\rho
\quad \text{as}\ z\to\infty, \; z\in S_j' .
\end{equation}
Replacing, without changing notation, $S_j'$ by a smaller subsector we find that
\begin{equation}\label{13c}
\frac{F_0'(z)}{F_0(z)\log F_0(z)}\sim - e_j \rho b^\rho z^{\rho -1}
\quad \text{as}\ z\to\infty, \; z\in S_j' .
\end{equation}
Hence
\begin{equation}\label{13d}
\begin{aligned}
E(z)
&=\frac{F_0(z)}{F_0'(z)}
\sim - \frac{e_j \rho b^\rho z^{\rho -1}}{\log F_0(z)}
\quad\text{as}\ z\to\infty,\ z\in S_j'
\end{aligned}
\end{equation}
so that $\log E(z) \sim -\log\log F_0(z)$ and hence
\begin{equation}\label{7l}
\log E(z) 
\sim e_j(bz)^\rho
\quad\text{as}\ z\to\infty,\ z\in S_j'.
\end{equation}

Let now $S_j$ be a large sector and let $S_j'$ be a subsector which is mapped by $p_j$
to a subsector of the first quadrant $\{z\colon \re z>0,\,\im z>0\}$.
By~\eqref{12c} we have
\begin{equation}\label{13e}
\begin{aligned}
F_0(z)-a_j
&= \log B_j(p_j(\beta(z))))-a_j
\sim  c p_j(\beta(z))^d \exp(-p_j(\beta(z)))
\quad \text{as}\ z\to\infty, \; z\in S_j'
\end{aligned}
\end{equation}
Similarly as above this yields 
\begin{equation}\label{13f}
\log ( F_0(z)-a_j) \sim -p_j(\beta(z))
\sim -e_j(bz)^\rho
\quad \text{as}\ z\to\infty, \; z\in S_j'
\end{equation}
and thus, passing to a smaller subsector,
\begin{equation}\label{13g}
\frac{F_0'(z)}{F_0(z)-a_j} \sim - e_j \rho b^\rho z^{\rho -1}
\quad \text{as}\ z\to\infty, \; z\in S_j' .
\end{equation}
We conclude that 
\begin{equation}\label{13h}
E(z)=\frac{F_0(z)}{F_0'(z)}\sim \frac{a_j}{F_0'(z)}
\sim -\frac{1}{ e_j \rho b^\rho z^{\rho -1} (F_0(z)-a_j)}
\quad \text{as}\ z\to\infty, \; z\in S_j'.
\end{equation}
Thus $\log E(z)\sim -\log ( F_0(z)-a_j)$ and~\eqref{13f} yields that~\eqref{7l}
holds again.

The case that $S_j'$ is mapped by $p_j$
to a subsector of the fourth quadrant $\{z\colon \re z>0,\,\im z<0\}$
 is analogous. Then the above equations hold with $a_j$ 
replaced by $a_{j+1}$, and again we obtain~\eqref{7l}.

Next we consider the case that $S_j$ is a small sector.
Again, let $S_j'$ be a subsector which is mapped by $p_j$
to a subsector of the first quadrant. Then
\begin{equation}\label{7t}
F_0(z)= 
\tan\!\left(\frac{p_j(\beta(z))}{2}\right) \im a_{j+1} +\re a_{j+1}
\quad\text{for}\ z\in S_j' ,
\end{equation}
provided $|z|$ is sufficiently large.
We conclude that
$F_0(z) \to a_{j+1}$ as $z\to\infty$, $z\in S_j'$,
and passing as before without change of notation to a smaller sector,
\begin{equation}\label{7v}
F_0'(z)
=\frac{p_j'(\beta(z)) \beta'(z)}{2\cos^2\!\left(\frac{p_j(\beta(z))}{2}\right)} \im a_{j+1}
\sim \frac{e_j \rho b^\rho z^{\rho-1} \im a_{j+1}}{2\exp(-i p_j(\beta(z)))}    
\quad\text{as}\ z\to\infty,\ z\in S_j' .
\end{equation}
It follows that 
\begin{equation}\label{7w}
E(z)\sim \frac{2 a_{j+1}\exp (-i p_j(\beta(z)) }{e_j \rho b^\rho z^{\rho-1}\im a_{j+1}}
\end{equation}
and hence
\begin{equation}\label{7x}
\log E(z) \sim  - i p_j(\beta(z)) \sim - i e_j (bz)^\rho
\quad\text{as}\ z\to\infty,\ z\in S_j'.
\end{equation}
An analogous argument shows that 
if $S_j'$ is mapped by $p_j$ to a subsector of the fourth quadrant,
then
\begin{equation}\label{7y}
\log E(z) \sim  i p_j(\beta(z)) \sim i e_j(bz)^\rho
\quad\text{as}\ z\to\infty,\ z\in S_j'.
\end{equation}

Suppose now that we are in case $(ii)$. 
Then~\eqref{7l} holds for $1\leq j\leq m-1$.
If $j=0$ and $S_0'$ is a subsector of $S_0$ which is contained in the upper half-plane,
then $p_j(S_0')$ is contained in the first quadrant and thus 
we have~\eqref{7x} with $j=0$.
Recalling that $e_0=1$ and $e_j=-i$ for all other $j$ in case $(ii)$, 
we find that if $T$ is any closed subsector of the upper half-plane
whose image under $z\mapsto z^\rho$ does not intersect the real or imaginary axis, then
\begin{equation}\label{8a}
\log E(z) \sim - i b^\rho z^\rho
\quad\text{as}\ z\to\infty,\ z\in T .
\end{equation}
With $c=b^\rho$ this yields that 
\begin{equation}\label{8b}
\begin{aligned}
\log |E(z)|
& \sim \re\!\left(- i b^\rho r^\rho e^{i\rho t}\right)
= c r^\rho \cos\!\left(\rho t-\frac{\pi}{2}\right)
= c r^\rho \sin(\rho t)
\quad\text{as}\ r\to\infty,\ re^{it}\in T.
\end{aligned}
\end{equation}
Since $E$ is symmetric, an analogous result holds for subsectors of the lower half-plane.
Thus $E$ is of completely regular growth on every ray except for finitely many.
Since the set of rays of completely regular growth is closed~\cite[\S~III.1]{L},
$E$ is of completely regular growth in the plane, with indicator 
as stated.

An analogous reasoning can be made in case $(iii)$.
In this case a subsector of $S_m$ which is contained in the upper half-plane 
is mapped to the fourth quadrant. Thus we have to use~\eqref{7y} instead of~\eqref{7x}
if $j=m/2$. But since $e_m=-1$ we again find that~\eqref{8a} holds for any subsector $T$ 
of the upper half-plane whose image under $z\mapsto z^\rho$ does not intersect
the real or imaginary axis.
As before we can conclude that $E$ has completely regular growth, with indicator 
as stated.

Finally, to prove that $A$ has completely regular growth, 
we note that it follows from~\eqref{8a} that if $T$ is a closed sector
containing no zeros of $E$, then there exists a constant $c'$ such that 
\begin{equation}\label{8e}
-2\frac{E''(z)}{E(z)}+ \left(\frac{E'(z)}{E(z)}\right)^2 \sim c' z^{2\rho-2}
\quad\text{as}\ z\to\infty,\ z\in T .
\end{equation}
Since $E$ has completely regular growth this implies together with~\eqref{AE} that
$A$ has completely regular growth, 
with indicator given by~\eqref{h_A}. \qed

\begin{remark}\label{remark10}
It follows from~\eqref{8e} that if $E$ has infinitely many positive zeros, then
\begin{equation}\label{14a}
A(z)\sim \frac14 c'z^{2\rho-2}
\quad\text{as}\ z\to\infty,\ z\in T ,
\end{equation}
for any closed subsector $T$ of $S_0\setminus \R$.
Lemma~\ref{lemma-lindeloef},
applied to $A(z)/z^{2\rho-2}$,
shows that~\eqref{14a} in fact holds for any closed subsector $T$ of $S_0$.
An analogous result holds if $E$ has infinitely many negative zeros.

Thus we actually have a much more precise description of the asymptotics of $A$
than given by~\eqref{h_A}: In the sectors corresponding to the intervals 
where $h_A=0$ we have~\eqref{14a}.

In particular is follows from~\eqref{14a} that $A$ is non-constant.
But this can also be deduced directly from the hypothesis that $F^{-1}$ has infinitely
many singularities over $0$ or~$\infty$.
\end{remark}

\begin{remark}\label{remark11}
Suppose that $E$ has infinitely many positive zeros.
Since between two positive zeros of a solution $w$ of~\eqref{1}
there is positive local maximum or a negative local minimum of~$w$,
we deduce from~\eqref{1} that $c'>0$ in~\eqref{14a}.
This implies that $A'(x)>0$ for all large positive~$x$. 
It follows (see~\cite{Biernacki1933} or~\cite[Chapter XIV, Part~I, Theorem~3.1]{Hartman1973})
that all solutions of~\eqref{1} are bounded on the positive real axis.
In particular, $E$ is bounded there.
Alternatively, this can be obtained from 
$F_0(z)=T_{a_1}(p_j(\beta(z))$, which holds for $z$ of sufficiently
large modulus in any subsector of $S_0$.
Again an analogous result holds if $E$ has infinitely many negative zeros.

Since $E'$ has only finitely many non-real critical points by~Lemma~\ref{lemma6},
this yields that
the set of critical values of $E$ is bounded. Moreover, by the Denjoy-Carleman-Ahlfors
theorem~\cite[Chapter~5, \S~1]{GO}, $E$ has only finitely many asymptotic values. We conclude that 
$E$ is in the class $B$ consisting of all entire functions for which
the set of critical and (finite) asymptotic values is bounded.
It plays an important role in value distribution and
holomorphic dynamics~\cite{6}, as does the Speiser class.
\end{remark}

\section{Proof of Theorem~\ref{theorem4}} \label{proof-thm4}
We will use the following result~\cite[Theorem~1]{Bergweiler2009}.
\begin{lemma}\label{lemma-ldv}
Let $F$ be a meromorphic function such that the
preimage of three points belongs to the real line.
Then $F$ maps the real line into a circle, unless
\begin{equation}\label{9a}
F(z)=L\left(\frac{1-e^{i(a_1z-b_1)}}{1-e^{i(a_2z-b_2)}}\right),
\end{equation}
where $L$ is a linear-fractional transformation and $a_j,b_j\in\R$.
\end{lemma}
Let $\D$ be the unit disk.
A meromorphic function  $F\colon\D\to\bC$ is called
\emph{normal} if the family $\{F\circ S\colon S\in\operatorname{Aut}(\D)\}$
is normal, where $\operatorname{Aut}(\D)$ denotes the set of biholomorphic
maps from $\D$ to~$\D$.
The following result is due to Lehto and Virtanen~\cite[Theorem~2]{Lehto1957}.
\begin{lemma}\label{lemma-lv}
Let $F\colon \D\to\bC$ be a normal meromorphic function.
Suppose that there exist $a\in\bC$ and a curve $\gamma$ ending at point
$P\in\partial\D$
such that $F(z)\to a$ as $z\to P$, $z\in\gamma$.
Then $f$ has the angular limit $a$ at~$P$.
\end{lemma}
\begin{proof}[Proof of Theorem~\ref{theorem4}]
Let $w_1,w_2,w_3$ be pairwise linearly independent
solutions of~\eqref{1} with only real zeros.
Without loss of generality we may assume that $w_3=w_1-w_2$, since otherwise
we can replace $w_1$ and $w_2$ by suitable multiples.
As before we put $F=w_2/w_1$. Then $F$ is a locally univalent meromorphic
function which has only real zeros, $1$-points and poles.

In view of~\eqref{Sch} we have to show that the Schwarzian derivative
of~$F$ is constant.
This is the case if $F$ has the form~\eqref{9a}.
Using Lemma~\ref{lemma-ldv} we may thus assume that $F$  maps $\R$ to a circle
$C$.

The Schwarzian derivative of a linear-fractional transformation is $0$.
So we may assume that $F$ is not a linear-fractional transformation.
Then the inverse of $F$ has a singularity. Since $F$ is locally univalent,
this means that $F$ has an asymptotic value~$a$. Let $\gamma$ be an
asymptotic path for~$a$.
Then $\overline{\gamma}$ is an asymptotic path for~$a^*$, the point
symmetric to $a$ with respect to the circle $C$.
If $\gamma$ crosses the real axis infinitely often, then $a=a^*$. 
In this case we can build from $\gamma$ and $\overline{\gamma}$
an asymptotic path which is contained in the upper half-plane~$H^+$.
If $a\neq a^*$ is non-real, then (the tail of) one of the curves $\gamma$
or
$\overline{\gamma}$ is in $H^+$ anyway, and we may assume without loss of
generality that this holds for~$\gamma$. Thus $F$ has the asymptotic value
$a$ with an asymptotic path in the upper half-plane.

As $F$ omits the values $0$, $1$ and $\infty$ in the upper half-plane, $F$ is
normal there.
Lemma~\ref{lemma-lv} yields that $F(z)\to a$ as $|z|\to\infty$,
$\varepsilon<\arg z<\pi-\varepsilon$.
In particular, $F$ has only one asymptotic value which has an asymptotic
path in the upper half-plane.
Overall we see that $F$ has at most two asymptotic values,
namely $a$ and~$a^*$.
Thus the inverse of $F$ has at most two singularities.
In fact, this also shows that $a\not\in C$ so that $a\neq a^*$,
since otherwise $F^{-1}$ would have only one
singularity, which is only possible for a linear-fractional map. 
We conclude that $F\colon\C\to\bC\setminus\{a,a^*\}$ is a covering.
With $L(z)=(z-a)/(z-a^*)$ we deduce that
$L\circ F\colon\C\to\C^*$ is a covering. Thus there exists $c\in\C^*$ and
$d\in\C$ such that $(L\circ F)(z)=\exp(cz+d)$.
It follows that the Schwarzian of $L\circ F$ and hence of $F$ is constant.
Thus $A$ is constant.
\end{proof}

\end{document}